\documentclass[12pt,a4paper]{amsart}
\usepackage{amsfonts}
\usepackage{amsthm}
\usepackage{amsmath}
\usepackage{amscd}
\usepackage[latin2]{inputenc}
\usepackage{t1enc}
\usepackage[mathscr]{eucal}
\usepackage{indentfirst}
\usepackage{graphicx}
\usepackage{graphics}
\usepackage{pict2e}
\usepackage{xcolor}
\usepackage{hyperref}
\usepackage{dsfont}
\newcommand{\E}{\mathbb{E}}

\hypersetup{
    colorlinks=true,
    linkcolor=blue,
    urlcolor=blue,
    linktoc=all
}

\usepackage{epic}
\numberwithin{equation}{section}
\usepackage[margin=2.5cm]{geometry}
\usepackage{epstopdf} 
\usepackage{cancel}
\usepackage{amsmath}
\usepackage{mathrsfs}
\usepackage{amsfonts}
\usepackage{dsfont}
\usepackage{amssymb}
\usepackage{graphicx}
\usepackage{tabu}
\usepackage{blindtext}
\usepackage{amsthm}
\usepackage{tikz}
\usepackage{relsize}
\usepackage{amsthm}
\usepackage{yhmath}
\usepackage[normalem]{ulem} 

\theoremstyle{plain}
\newtheorem{thm}{Theorem}[section]
\newtheorem{lemma}[thm]{Lemma}

\newtheorem{prop}[thm]{Proposition}
\newtheorem{cor}[thm]{Corollary}

\theoremstyle{definition}

\newtheorem{definition}[thm]{Definition}
\newtheorem{remark}[thm]{Remark}

\theoremstyle{remark}

\newtheorem*{stat*}{Statement}
\newtheorem*{prop*}{Proposition}
\numberwithin{equation}{section}

\newcommand{\N}{\mathbb{N}}

\newcommand{\R}{\mathbb{R}}

\newcommand{\cov}{\operatorname{cov}}

\newcommand{\var}{\text{Var}}

\newcommand{\of}[1]{\left(#1\right)}

\newcommand{\overbar}[1]{\mkern 1.5mu\overline{\mkern-1.5mu#1\mkern-1.5mu}\mkern 1.5mu}

\makeatletter

\def\chaptermark#1{}

\def\chapter{%
  \if@openright\cleardoublepage\else\clearpage\fi
  \thispagestyle{plain}\global\@topnum\z@
  \@afterindenttrue \secdef\@chapter\@schapter}

\def\@chapter[#1]#2{\refstepcounter{chapter}%
  \ifnum\c@secnumdepth<\z@ \let\@secnumber\@empty
  \else \let\@secnumber\thechapter \fi
  \typeout{\chaptername\space\@secnumber}%
  \def\@toclevel{0}%
  \ifx\chaptername\appendixname \@tocwriteb\tocappendix{chapter}{#2}%
  \else \@tocwriteb\tocchapter{chapter}{#2}\fi
  \chaptermark{#1}%
  \addtocontents{lof}{\protect\addvspace{10\p@}}%
  \addtocontents{lot}{\protect\addvspace{10\p@}}%
  \@makechapterhead{#2}\@afterheading}
\def\@schapter#1{\typeout{#1}%
  \let\@secnumber\@empty
  \def\@toclevel{0}%
  \ifx\chaptername\appendixname \@tocwriteb\tocappendix{chapter}{#1}%
  \else \@tocwriteb\tocchapter{chapter}{#1}\fi
  \chaptermark{#1}%
  \addtocontents{lof}{\protect\addvspace{10\p@}}%
  \addtocontents{lot}{\protect\addvspace{10\p@}}%
  \@makeschapterhead{#1}\@afterheading}
\newcommand\chaptername{Chapter}

\def\@makechapterhead#1{\global\topskip 7.5pc\relax
  \begingroup
  \fontsize{\@xivpt}{18}\bfseries\centering
    \ifnum\c@secnumdepth>\m@ne
      \leavevmode \hskip-\leftskip
      \rlap{\vbox to\z@{\vss
          \centerline{\normalsize\mdseries
              \uppercase\@xp{\chaptername}\enspace\thechapter}
          \vskip 3pc}}\hskip\leftskip\fi
     #1\par \endgroup
  \skip@34\p@ \advance\skip@-\normalbaselineskip
  \vskip\skip@ }
\def\@makeschapterhead#1{\global\topskip 7.5pc\relax
  \begingroup
  \fontsize{\@xivpt}{18}\bfseries\centering
  #1\par \endgroup
  \skip@34\p@ \advance\skip@-\normalbaselineskip
  \vskip\skip@ }
\def\appendix{\par
  \c@chapter\z@ \c@section\z@
  \let\chaptername\appendixname
  \def\thechapter{\@Alph\c@chapter}}

\newcounter{chapter}

\newif\if@openright

\renewcommand{\tocsection}[3]{%
  \indentlabel{\@ifnotempty{#2}{\bfseries\ignorespaces#1 #2\quad}}\bfseries#3} 
\renewcommand{\tocsubsection}[3]{%
  \indentlabel{\@ifnotempty{#2}{\ignorespaces#1 #2\quad}}#3}

\newcommand\@dotsep{4.5}
\def\@tocline#1#2#3#4#5#6#7{\relax
  \ifnum #1>\c@tocdepth 
  \else
    \par \addpenalty\@secpenalty\addvspace{#2}%
    \begingroup \hyphenpenalty\@M
    \@ifempty{#4}{%
      \@tempdima\csname r@tocindent\number#1\endcsname\relax
    }{%
      \@tempdima#4\relax
    }%
    \parindent\z@ \leftskip#3\relax \advance\leftskip\@tempdima\relax
    \rightskip\@pnumwidth plus1em \parfillskip-\@pnumwidth
    #5\leavevmode\hskip-\@tempdima{#6}\nobreak
    \leaders\hbox{$\m@th\mkern \@dotsep mu\hbox{.}\mkern \@dotsep mu$}\hfill
    \nobreak
    \hbox to\@pnumwidth{\@tocpagenum{\ifnum#1=1\bfseries\fi#7}}\par
    \nobreak
    \endgroup
  \fi}
\AtBeginDocument{%
\expandafter\renewcommand\csname r@tocindent0\endcsname{0pt}
}
\def\l@subsection{\@tocline{2}{0pt}{2.5pc}{5pc}{}}
\makeatother

\usepackage{lipsum}

\begin{document}

\title{Pair Dependent Linear Statistics for C$\beta$E}

\author[A. Aguirre A. Soshnikov J. Sumpter]{Ander Aguirre, Alexander Soshnikov and Joshua Sumpter}

\address{University of California at Davis \\ Department of Mathematics \\ 1 Shields Avenue \\Davis CA 95616 \\ United States of America} 

\email{aaguirre@ucdavis.edu}

\address{University of California at Davis \\ Department of Mathematics \\ 1 Shields Avenue \\  Davis CA 95616 \\ United States of America} 

\email{soshniko@math.ucdavis.edu}

\address{University of California at Davis \\ Department of Mathematics \\ 1 Shields Avenue \\  Davis CA 95616 \\ United States of America} 

\email{jsumpter@math.ucdavis.edu}

 \subjclass[2010]{Primary: 60F05. }

 \keywords{Random Matrices, Central Limit Theorem} 

\date{December 17, 2019}

\begin{abstract}

We study the limiting distribution of 
a pair counting statistics of the form \\
$ \sum_{1\leq i\neq j\leq N} f(L_N\*(\theta_i-\theta_j))$
for the circular $\beta$-ensemble (C$\beta$E) of 
random matrices for sufficiently smooth test function $f$ and $L_N=O(N).$
For $\beta=2$ and $L_N=N$ our results are inspired by a classical result of Montgomery on pair correlation of zeros of Riemann zeta function.
\end{abstract}
\maketitle

\tableofcontents

\section{Introduction}

        Let $\{\theta_i\}_{i=1}^{N}$ be distributed on the unit circle according the circular $\beta$-ensemble (C$\beta$E), 
i.e. have joint probability density
            \begin{align} \label{betaensemble}
                p_N^\beta(\overline{\theta})=\frac{1}{Z_{N,\beta}}\prod_{1\leq j< k\leq N}\left|e^{i\theta_j}-e^{i\theta_k}\right|^\beta, \ \ \ 
0\leq \theta_1, \ldots, \theta_N <2\*\pi, 
            \end{align}
        where $\beta>0$ and  $Z_{N,\beta}$ is an appropriate normalization constant that can be explicitly written in terms of the Gamma function as 
follows:
            \begin{align}
\label{selberg}
                \frac{Z_{N,\beta}}{(2\pi)^N}= 
\frac{1}{(2\pi)^N}\int_{\mathbb{T}^N} \prod_{1\leq j< k\leq N}\left|e^{i\theta_j}-e^{i\theta_k}\right|^\beta d\overline{\theta}= 
\frac{\Gamma\left(1+\frac{\beta N}{2}\right)}{\Gamma\left(1+ \frac{\beta}{2}\right)^N}.
       \end{align}

The ensemble was introduced in Random Matrix Theory by Dyson in \cite{Dyson1}-\cite{Dyson3}. Three special cases $\beta=1,2$ and $4$ correspond to the 
so-called Circular Orthogonal Ensemble (COE), Circular Unitary Ensemble (CUE), and Circular Symplectic Ensemble (CSE). For the CUE ($\beta=2,$)
(\ref{betaensemble}) is the joint distribution of the eigenvalues of an $n\times n$ random unitary matrix U distributed according to the Haar measure.
The joint distribution of $U^t\*U$ gives the COE. If $U^D$ denotes the quaternion dual, then $U^D\*U$ gives the CSE for even $n$. We refer the reader to 
\cite{mehta} for details. For arbitrary $\beta>0,$ a (sparse) random matrix model with eigenvalues distribution following (\ref{betaensemble}) 
was introduced in \cite{KN}. 

Since the probability density (\ref{betaensemble})  is invariant under rotations, one is interested in the fluctuation of the empirical spectral density
around the Lebesgue measure on the unit circle. For results on the limiting distribution of linear statistics $\sum_{j=1}^n f(\theta_j) $ for 
``sufficiently nice'' test functions $f$ we refer the reader to \cite{johansson1}, \cite{DS}, \cite{johansson3}, \cite{Rains}, \cite{sasha}, \cite{HKOC}, 
\cite{DE}, \cite{WF1}, \cite{WF2}, \cite{webb}, \cite{lambert}, \cite{FTW}.  A significant part of literature is devoted to statistical properties
of spectral linear statistics in the mesoscopic regime $\sum_{j=1}^n f(L_N\*\theta_j), \ 1\ll L_N \ll N$ (see e.g. \cite{sasha}, \cite{lambert}, \cite{BL},
\cite{HK}, \cite{LSX}, \cite{LS} and references therein for this and related problems).

Denote the Fourier coefficients of an $L^2(\mathbb{T})$ function $f$ as
\begin{align}
\label{FourierS}
\hat{f}(m)=\frac{1}{2\*\pi}\*\int_0^{2\*\pi}f(x)\* e^{-i\*m\*x}\* dx.
\end{align}

It was proven by Johansson in \cite{johansson1} that for arbitrary $\beta>0$ and sufficiently smooth real-valued $f$
\[
\frac{\sum_{j=1}^N f(\theta_j) - N\*\hat{f}(0)}{\sqrt{\frac{2}{\beta}\*\sum_{-\infty}^{\infty} |\hat{f}(m)|^2\*|m|}}
\]
converges in distribution to a standard Gaussian random variable.

If $f$ is not smooth enough and the variance of the linear statistic goes to infinity with $N,$ Diaconis and Evans \cite{DE} 
proved the CLT in the case 
$\beta=2$ provided the sequence $\{\sum_{-n}^{n} |\hat{f}(m)|^2\*|m| \}_{n \in \N}$ is slowly varying.
For the results about the Gaussian fluctuation of the number of eigenvalues in arcs we 
refer the reader to \cite{FTW} and references therein.  For the results on the characteristic polynomial of a random unitary matrix, we refer the reader 
to \cite{HKOC}, \cite{BF}.

This paper is devoted to studying the limiting distribution of pair counting functions

            \begin{align}
\label{pairs}
                S_N(f)=\sum_{1\leq i\neq j\leq N} f(L_N\*(\theta_i-\theta_j)_c),
            \end{align}
where $(\theta_i-\theta_j)_c$ is the phase difference on the unit circle, i.e.
\begin{align}
\label{circlediff}
(\theta-\phi)_c= \begin{cases}  \theta-\phi&\text{if }  -\pi\leq \theta-\phi<\pi,\\
 \theta-\phi -2\*\pi  &\text{if  }  \pi\leq \theta-\phi<2\*\pi,\\
\theta-\phi +2\*\pi&\text{if}  -2\*\pi<\theta-\phi<-\pi,
\end{cases}
\end{align}

$f$ is a a smooth function, and $L_N/N$ is bounded from above. The case $\beta=2, \ L_N=N$ is of the main interest since it is motivated by
a classical result of Montgomery on pair correlation of zeros of the Riemann zeta function \cite{montgomery1}-\cite{montgomery2}. Assuming the Riemann 
Hypothesis (RH), Montgomery studied the distribution of the ``non-trivial'' zeros $\{ 1/2 \pm \gamma_n\}, \ \gamma_n$ real positive.  Rescaling zeros
\[ 
\tilde{\gamma}_n= \frac{\gamma_n}{2\*\pi} \*\log(\gamma_n),
\]
Montgomery essentially studied the statistic
\[
F(\alpha)=T^{-1}\* \sum_{0<\tilde{\gamma}_j, \tilde{\gamma_k}\leq T} \exp(i\*\alpha\*(\tilde{\gamma}_j-\tilde{\gamma}_k)) 
\*\frac{4}{4+(\tilde{\gamma}_j-\tilde{\gamma}_k)^2/\log(T)^2},
\]
for real $\alpha$ and large real $T.$ Assuming RH, Montgomery rigorously proved that for \\$ 0\leq \alpha\leq 1 $ and large $T$ the statistic behaves as
\[ (1+o(1))\* T^{-2\*\alpha}\*\log(T)^{2\alpha-1} +\alpha +o(1),\]
He also proved heuristic arguments that
\[
F(\alpha)=1+o(1),
\]
for $\alpha\geq 1,$ uniformly in bounded intervals,
which allowed him to conjecture that rescaled non-trivial zeros of the Riemann zeta function behave locally as the rescaled eigenvalues of the CUE.
It should be noted that $\alpha \mapsto \min(|\alpha|, 1)$ is the Fourier transform of $\delta(x)-\left(\frac{\sin(\pi\*x)}{\pi\*x}\right)^2,$ 
which suggests 
that rescaled two-point correlations of zeros of Riemann zeta functions and eigenvalues of a large random unitary matrix coincide in the limit (we 
refer the reader to \cite{RS} for further developments.)
Hence comes our interest in studying the limiting distribution of (\ref{pairs}), especially in the microscopic regime $L_N=N.$
We prove Gaussian fluctuation under certain technical conditions in the mesoscopic and microscopic cases (see Theorems 2.4 and 2.5 below). The 
fluctuation is not Gaussian in the macroscopic case provided a test function $f$ is sufficiently smooth (see Theorem 2.1 below).

In the microscopic case, we note that 
even though the number of terms in 
\begin{align}
\label{localcirc}
S_N(f)=\sum_{1\leq i\neq j\leq N} f(N\*(\theta_i-\theta_j)_c) 
\end{align}
is proportional to $N^2,$ the number of non-zero terms in the sum is of order $N$ provided a test function $f$ decays sufficiently fast at infinity.
We note that the limiting fluctuation of the sum (\ref{localcirc}) does not change if one replaces the circular difference (\ref{circlediff}) in the argument of $f(N\cdot)$
by the regular one and studies instead
\begin{align}
\label{localnoncirc}
\sum_{1\leq i\neq j\leq N} f(N\*(\theta_i-\theta_j)), 
\end{align}
since the number of pairs of the eigenvalues in a $O(N^{-1})$ neighborhood of $\theta=0$ is bounded in probability.

The pairs $(i,j)$ that give non-zero contribution to (\ref{localcirc}) correspond to neighbors (nearest neighbors, next-to-nearest-neighbors, etc.)
Ordering the particles $\{\theta_j\}_{j=1}^{N}$ we introduce order statistics
\[0\leq \theta_{(1)}<\theta_{(2)}<\ldots<\theta_{(N)}<2\*\pi,\]
and rescaled nearest-neighbor spacings 
\[
\tau_j=N\*(\theta_{(j+1)}-\theta_{(j)}), \ \ j=1, \ldots,N-1.
\]
One can compare (\ref{localcase}) with the sum
\begin{align}
\label{spacings-sum}
\mathcal{A}_N(f):=\sum_{j=1}^{N-1} f(\tau_j).
\end{align}
The empirical distribution function of nearest-neighbor spacings was studied in \cite{Sasha}, where it was shown that
\[\xi_N(s)=\left(\#\{j: \tau_j\leq s\}-\E \#\{j: \tau_j\leq s\}\right)\*N^{-1/2}
\]
converges in finite-dimensional distributions and also, after minor modifications, in functional sense, to a Gaussian random process
$\xi(s)$ as $N\to \infty$ (see Theorems 1.1 and 1.2 in \cite{Sasha}).
As a corollary, 
\[\left(\sum_{j=1}^{N-1} f(\tau_j)-\E\sum_{j=1}^{N-1} f(\tau_j)\right)\*N^{-1/2} = \int f(s)\*d \xi_N(s)=-\int f'(s) \* \xi_N(s) \*ds\]
converges in distribution a Gaussian random variable 
$-\int f'(s) \* \xi(s)\*ds. $

The paper is organized as follows. We formulate our results in the next section. The unscaled case $(L_N=1)$ is studied in Sections 3 and 4. We discuss 
joint cumulants for linear statistics in the $\beta=2$ case in Section 5. The mesoscopic case $L_N\to \infty, L_N/N\to 0$ is studied in Section 6.
The microscopic case $L_N=N, \ \ \beta=2$ is studied in Section 7. The auxiliary results are collected in Appendices 1-3.

The notation $a_N=O(b_N)$ means 
that the ratio $ a_N/b_N$ is bounded from above in absolute value. The notation $a_N=o(b_N)$ means that $a_n/b_N\to 0$ as $N\to \infty.$
Occasionally, for non-negative quantities, in this case we will also use the notation $a_N \ll b_N.$ Finally, we note that we use similar notations 
$\hat{f}(m)$ in (\ref{FourierS}) and $\hat{f}(t)$ in (\ref{FourierT}) for the Fourier series coefficients in of a function defined on a unit circle and 
for the Fourier transform of a function defined on the real line, correspondingly. The first notation is used when we study the global regime ($L_N=1$) and
the second notation is used in the mesoscopic and local regimes.

Research has been partially supported  by the Simons Foundation Collaboration Grant for Mathematicians \#312391.
            
    \section{Main Theorems}   
We start with the unscaled case $L_N=1.$    

        \begin{thm} Consider the C$\beta$E (\ref{betaensemble}) and let

            \begin{align}
\label{pairs1}
                S_N(f)=\sum_{1\leq i\neq j\leq N} f(\theta_i-\theta_j),
            \end{align}
            
where $f$ is a real even function on the unit circle such that $f'\in L^2(\mathbb{T})$ for $\beta=2$,
$\sum_{k\in \mathbb{Z}}|\hat{f}(k)||k|<\infty$ for $\beta<2, \ \sum_{k\in \mathbb{Z}}|\hat{f}(k)||k|\*\log(|k|+1)<\infty$ for $\beta=4,$ 
and \\
$\sum_{k\in \mathbb{Z}}|\hat{f}(k)||k|^2<\infty$ for $\beta \in (2,4)\cup (4, \infty).$

Then we have the following convergence in distribution as $N\rightarrow \infty$:

            $$S_N(f)-\E S_N(f)\xrightarrow{\hspace{2mm}\mathcal{D}\hspace{2mm}  } \frac{4}{\beta}\sum_{m=1}^{\infty}\hat{f}(m)m(\varphi_m-1),$$
        
        where $\varphi_m$ are i.i.d. exponential random variables with  $\E(\varphi_m)=1$.
        \end{thm}

\begin{remark}
The mathematical expectation of $S_N(f)$ satisfies
\begin{align}
\label{matozh}
\E S_N(f)= \hat{f}(0)\*N^2 -f(0)\*N +\frac{2}{\beta}\*\sum_{k=-\infty}^{\infty} \hat{f}(k)\*|k| +o(1)
\end{align}
for sufficiently smooth test function $f.$  In particular, for $\beta=2$ one has
\begin{align}
\label{matozh1}
\E S_N(f)= \hat{f}(0)\*N^2 -f(0)\*N +\sum_{k=-\infty}^{\infty} \hat{f}(k) \*\min(|k|,N). 
\end{align}
In general, (\ref{matozh}) holds for $\beta\leq 2$ under the optimal condition $\sum_{k=-\infty}^{\infty} |\hat{f}(k)| \*|k|<\infty.$
For $\beta>2$  we can show that $\sum_{k=-\infty}^{\infty} |\hat{f}(k)| \*|k|^2<\infty$ implies (\ref{matozh}).
\end{remark}

\begin{remark} It is reasonable to expect that CLT holds for $S_N(f)$ provided the series $\sum_{k\in \mathbb{Z}}|\hat{f}(k)|^2\*|k|^2$ 
diverges and the sequence of its
partial sums satisfies some regularity condition since the sum of independent random variables 
$\sum_{m=1}^{N}\hat{f}(m)m(\varphi_m-1)$ converges to a Gaussian distribution after normalization under mild assumptions on the coefficients $\hat{f}(m).$
This is outside the scope of this paper. The case of a slowly growing variance is considered in \cite{AS}.

Here we just note that 
for $f= (1/2)\ln|2\sin(\theta/2)|$ and arbitrary $\beta>0$
        \[
            \frac{S_N(f)- E(S_N(f))}{\sqrt{N}}\longrightarrow\mathcal{N}\left(0, \frac{2-\beta\Psi^{(2)}\left(1+\frac{\beta}{2}\right)}{4\beta}\right),
        \]
    where $\Psi^{(k)}(x)= \left(\frac{d}{dx}\right)^k \log(\Gamma(x))$.  This is a simple corollary of the Selberg integral formula (\ref{selberg}).
\end{remark}

Now we consider the mesoscopic regime $1 <<L_N<<N.$
Let $f \in C^{\infty}_c(\R)$ be an even, smooth, compactly supported function on the real line. 
When $N$ is sufficiently large, the support of $f(L_N\*\theta)$ is contained in the interval $[-\pi,\pi]$. Extend $f(L_N\cdot)$ $2\*\pi$-periodically to the whole real line. 
Consider the random variable defined above in (\ref{pairs}), namely
\[
S_N(f(L_N\cdot))=\sum_{1\leq i\neq j\leq N} f(L_N\*(\theta_i-\theta_j)_c).
\]
Denote by
\begin{align}
\label{FourierT}
\hat{f}(t)=\frac{1}{\sqrt{2\*\pi}}\*\int_{\R} f(x)\*e^{-i\*t\*x} \* dx
\end{align}
the Fourier transform of $f$.
The following result holds.

 \begin{thm}
Let $f \in C^{\infty}_c(\R)$ be an even, smooth, compactly supported function on the real line.  Assume that
$1 <<L_N<<N,$ for $\beta =2$ and that $L_N$ grows to infinity slower than any positive power of $N$ for $\beta \neq 2.$ Then 
$(S_N(f(L_N\cdot)) -\E S_N(f(L_N\cdot)))\*L_N^{-1/2}$ converges in distribution to centered real Gaussian random variable with the variance
\[
\frac{4}{\pi\*\beta^2}\* \int_\mathbb{R} |\hat{f}(t)|^2 \*t^2\* dt.
\]
\end{thm}

Finally, we consider the local case $L_N=N.$ We establish the following CLT for $\beta =2.$

\begin{thm}
Let $f \in C^{\infty}_c(\R)$ be an even, smooth, compactly supported function on the real line. Consider
\begin{align}
\label{localcase}
S_N(f(N\cdot))=\sum_{1\leq i\neq j\leq N} f(N\*(\theta_i-\theta_j)_c).
\end{align}
Then $(S_N(f(N\cdot)) -\E S_N(f(N\cdot)))\*N^{-1/2}$ converges in distribution to centered real Gaussian random variable with the variance
\begin{align}
\label{vvvv}
\frac{1}{\pi}\* \int_{\mathbb{R}} |\hat{f}(t)|^2 \*\min(|t|,1)^2\* dt -\frac{1}{\pi}\* \int_{|s-t|\leq 1, |s|\vee|t|\geq 1} \hat{f}(t)\*\hat{f}(s)\*
(1-|s-t|)\* ds\*dt\\
-\frac{1}{\pi}\* \int_{0\leq s,t\leq 1, s+t>1} \hat{f}(s)\*\hat{f}(t)\*(s+t-1) \*ds\*dt. \nonumber
\end{align}
\end{thm}
\begin{remark}
As we noted in Section 1, the limiting distribution of (\ref{localcase}) does not change if one replaces the circular difference (\ref{circlediff}) in the argument of $f(N\cdot)$
by the regular one and studies a pair counting statistic (\ref{localnoncirc}) instead
since the number of pairs of the eigenvalues in a $O(N^{-1})$ neighborhood of $\theta=0$ is bounded in probability.
\end{remark}

\section{Proof of Theorem 2.1}

For trigonometric polynomials, Theorem 2.1 follows from the Johansson's CLT for linear statistics \cite{johansson1} and simple computations
in (\ref{formula}) below.
To prove the result for a wider class of test functions one needs variance bounds and standard $\epsilon/3$ type arguments 
(for the convenience of the reader, presented in Appendix 1.)
The proof  under the optimal condition on $f$ for $\beta=2$ requires careful variance computations given in Section 4.
\begin{proof}$\\$
Consider an even real-valued test function $f.$ Then
        \begin{align}
\label{formula}
            S_N(f)=\sum_{1\leq i\neq j\leq N}f(\theta_i-\theta_j)=2\sum_{m=1}^{\infty}\hat{f}(m)\left|\sum_{j=1}^N\exp\of{im\theta_j}\right|^2+ 
\hat{f}(0)\*N^2-N\*f(0).
        \end{align}
        In particular, for even trigonometric polynomials $f_k$ of degree $k$ we have:

            $$S_{N}(f_k)=2\sum_{m=1}^{k}\hat{f_k}(m)\left|\sum_{j=1}^N\exp\of{im\theta_j}\right|^2+\hat{f_k}(0)\*N^2-Nf_k(0).$$
We recall that \cite{johansson1} gives convergence of the real and imaginary parts of  $\sum_{j=1}^N\exp\of{im\theta_j}$ to 
independent  random variables $\mathcal{N}(0, \frac{m}{\beta})$ as  $N\rightarrow \infty$. 
Since the absolute value squared of a standard complex Gaussian random variable is exponentially distributed, the result follows for trigonometric 
polynomials. 

For more general test functions $f$, we obtain the desired 
result by approximating $f$ by the partial sums $f_k$ of Fourier series and interchanging the limits in (\ref{formula}). In fact, 
for $\beta=2$ we are able to prove the result of Theorem 2.1 under the optimal condition \[\sum_{k\in \mathbb{Z}}|\hat{f}(k)|^2\*|k|^2<\infty.\]
To achieve it, we first carefully compute the variance of $S_N(f)$ for finite $N$ and show that ``error'' terms are negligible in the limit. 
This is done in Section 4. In particular, we will prove Proposition 4.1, Corollary 4.2, and Proposition 4.3 in the next section. 
Then a standard $\epsilon/3$- type argument finishes the proof 
(see the Appendix 1 for the details). 

For $\beta\neq 2,$ we replace the Chebyshev bound with a corresponding Markov bound and apply the asymptotics 
results of Jiang and Matsumoto \cite{jm} on the moments of traces. Again, we refer the reader to the Appendix 1 for the details.

\section{Variance Calculation for $\beta=2$}
This section is devoted to the computation and asymptotic analysis of the variance of the pair counting statistic $S_N(f)$ defined in (\ref{pairs1}).
The main results of the section are Proposition 4.1 and Proposition 4.3. We assume $\beta=2$ for the rest of the section.
    \begin{prop} Let $f$ be a real even function on the unit circle such that $f'\in L^2(\mathbb{T})$ and let $\beta=2.$ Then
        \begin{align*}
            \var(S_N(f)) =4\left(\sum_{1\leq s\leq N-1}s^2(\hat{f}(s))^2 + 
N^2\sum_{N\leq s} (\hat{f}(s))^2 - N\sum_{\substack{N\leq s}}(\hat{f}(s))^2\right)
        \end{align*}
        \[
            -4\left(\sum_{\substack{1\leq s,t \\ 1\leq |s-t|\leq N-1\\ N\leq \max(s,t)}}(N-|s-t|)\hat{f}(s)\hat{f}(t)\hspace{2mm}
            +\sum_{\substack{1\leq s,t\leq N-1\\N+1\leq s+t}}((s+t)-N) \hat{f}(s)\hat{f}(t)\right).
        \]
        \end{prop}
        As a corollary, we obtain:
        \begin{cor} Let $s,t\in\mathbb{Z}_{\geq 0}$ and $\{\theta_m\}_{m=1}^N$ be distributed according to $CUE(N)$. Then
            \begin{align*}
                    \cov\left(\left|\sum_{m=1}^N e^{is\theta_m}\right|^2,\left|\sum_{m=1}^N e^{it\theta_m}\right|^2\right)= 
                    \begin{cases} s^2, &\text{  } 1\leq s=t\leq N-1,\hspace{2mm} 2s\leq N,\\
                                  N+ s^2 -2s, &\text{  } 1\leq s=t\leq N-1,\hspace{2mm} N+1\leq 2s,\\                                  
                                  N(N-1) , &\text{  } N\leq s=t,\\
                                  |s-t|-N, &\text{  } 1\leq |s-t|\leq N-1,\hspace{2mm}N\leq \max(s,t),\\
                                  N-(s+t), &\text{  } 1\leq s\neq t\leq N-1, N+1\leq s+t,\\
                                  0,&\text{else.}
                    \end{cases}.
                \end{align*}
\end{cor}
We note that the above formula immediately extends to the case where either $s$ or $t$ is negative, 
since $\left|\sum e^{is\theta_m}\right|=\left|\sum e^{-is\theta_m}\right|$.
For a graphical representation of the covariance function, see the diagram below.
                \vspace{5mm}
                \begin{center}
                    \begin{figure}[h]
                        \begin{tikzpicture}[scale=0.57, transform shape]
                            \node[] (A) at (9,9) {$N(N-1)$};
                            \node[] (s) at (0,12) {s};
                            \node[] (t) at (12,0) {t};
                            \node[] (SAL) at (6,-1/2) {$N-1$};
                            \node[] (TAL) at (-3/4,6) {$N-1$};
                            \node[] (l1) at (2,10) {$0$};
                            \node[] (l2) at (10,2) {$0$};
                            \node[] (l4) at (3,1) {$0$};
                            \node[] (l5) at (1,3) {$0$};
                            \node[] (l3) at (1.5,1.5) {$s^2$};
                            \node[] (l6) at (9,6) {$|s-t|-N$};
                            \node[] (l7) at (6,9) {$|s-t|-N$};
                            \node[] (l8) at (4.75,3) {$N-(s+t)$};
                            \node[] (l9) at (3, 4.75) {$N-(s+t)$};
                            \node[] (l10) at (4,4) {$N+s^2-2s$};
    
                            \path [->] (0,0) edge node[left] {} (s);
                            \path [->] (0,0) edge node[left] {} (t);
        
                            \path [-] (6,0) edge node[left] {} (6,6);
                            \path [-] (0,6) edge node[left] {} (6,6);
                            \path [-] (6,0) edge node[left] {} (0,6);
        
                            \path [->] (6,0) edge node[left] {} (12,6);
                            \path [->] (0,6) edge node[left] {} (6,12);
        
                            \path [dashed] (6,6) edge node[left]  {} (A);
                            \path [->] (A) edge[dashed] node[right]  {} (12,12);
                            \path [dashed] (0,0) edge node[left] {} (l3);
                            \path [dashed] (l3) edge node[left] {} (l10);
                            \path [dashed] (l10) edge node[left] {} (6,6);
                        \end{tikzpicture}
                        \caption{$\cov(s,t)$}
                    \end{figure}
                \end{center}

We need next proposition to prove Theorem 2.1 under the optimal assumptions on the test function $f.$
        \begin{prop}Let $\beta=2$ and $f$ satisfy the conditions of Theorem 2.1, i.e. $f$  is an even real function such 
that $f'\in L^2(\mathbb{T}).$ Then 
            \begin{align*}
                        \var(S_N(f))= 4\sum_{1\leq k\leq N-1}k^2|\hat{f}(k)|^2+o(1).
            \end{align*}
        \end{prop}

First, we prove Proposition 4.1. The proof follows from quite straightforward, but somewhat tedious computations given below.

        \begin{proof}$\\$
         We may assume, without loss of generality, that $\hat{f}(0)=0$. Let $\rho_{N,k}(\overbar{\theta})$ be the $k$-point correlation 
functions for $\{\theta_j\}_{j=1}^N$ distributed according to $CUE(N)$. It is well known that CUE point correlation functions have 
determinantal structure (see e.g. \cite{mehta}). In particular, if $Q_N(x,y)$ is the kernel of the orthogonal projection on \\
$Span \{ \frac{1}{\sqrt{2\pi}}\*e^{i\*k\*x}, \ \ 0\leq k\leq N-1\},$ namely
            \begin{align}
\label{detform}
                Q_N(x, y)= \frac{1}{2\pi}\sum_{k=0}^{N-1} e^{i\*k\*(x-y)},
            \end{align}
        then 
            \[
                \rho_{N,k}(\theta_1,\ldots, \theta_k)= \det\mathlarger{\mathlarger{(}}Q_N(\theta_i,\theta_j)\mathlarger{\mathlarger{)}}_{1\leq i, j \leq k}.
            \]
        A simple computation using (\ref{detform}) and $\hat{f}(0)=0$ gives    
        \begin{align*}
            \E (S_N(f))&= \E \left(\sum_{0\leq i\neq j\leq N-1} f(\theta_i-\theta_j)\right)\\
            &= \int_{\mathbb{T}^2} f(\theta_1-\theta_2)\rho_{N,2}(\theta_1,\theta_2)d\theta_1d\theta_2\\
            &= \sum_{|k|<N}(|k|-N)\hat{f}(k).
        \end{align*}
          
    Furthermore, the variance of $S_N(f)$ is given by
        \begin{align}
\label{400}
             &\E ((S_N(f))^2)-(\E (S_N(f)))^2=\\
\label{401}
 &2\int_{\mathbb{T}^2} f^2(\theta_1-\theta_2)\rho_{N,2}(\theta_1,\theta_2)d\theta_1d\theta_2 \\
\label{402}
             &+ 4\int_{\mathbb{T}^3} f(\theta_1-\theta_2)f(\theta_2-\theta_3)\rho_{N,3}(\theta_1,\theta_2,\theta_3)d\theta_1d\theta_2d\theta_3\\
\label{403}
             &+ \int_{\mathbb{T}^4} f(\theta_1-\theta_2)f(\theta_3-\theta_4)\rho_{N,4}(\theta_1,\theta_2,\theta_3,\theta_4)
d\theta_1d\theta_2d\theta_3d\theta_4-(\E(S_N(f)))^2,
\end{align}
which can be rewritten as
\begin{align}
\label{4.0}
&\E ((S_N(f))^2)-(\E (S_N(f)))^2=\\
\label{4.1}
             &2N^2\left(\widehat{f^2}(0)-\sum_{|k|\leq N-1}|\hat{f}(k)|^2\right) \\
\label{4.2}
             &+ 4\sum_{0\leq j,k,l\leq N-1}\hat{f}(j-k)\hat{f}(k-l) - 2\sum_{0\leq j,k,l,m\leq N-1}\hat{f}(j-k)\hat{f}(k-l)\chi_{(j-m=k-l)}\\
\label{4.3}
             &-2N\sum_{|k|\leq N-1}\widehat{f^2}(k) + 2\sum_{|k|\leq N-1}|k|\widehat{f^2}(k)+2\sum_{|k|\leq N-1}|k|^2|\hat{f}(k)|^2.
        \end{align}

The transition from (\ref{400}-\ref{403}) to (\ref{4.0}-\ref{4.3}) relies on straightforward but somewhat tedious computations given below.
The expression (\ref{401}) is equal to
\begin{align}
\label{411}
2\*\left(N^2\*\widehat{f^2}(0)+\sum_{|k|<N} (|k|-N)\*\widehat{f^2}(k) \right).
\end{align}
Using  $\hat{f}(0)=0$ the expression (\ref{402}) can be rewritten as
\begin{align}
\label{412}
&4\*\int_{\mathbb{T}^3} (f\ast f)(\theta_1-\theta_3)\*\rho_{N,1}(\theta_2)\*\rho_{N,2}(\theta_1,\theta_3)d\theta_1 d\theta_2d\theta_3\\
\label{413}
&4\*\int_{\mathbb{T}^3} f(\theta_1-\theta_2)\* f(\theta_2-\theta_3)\*Q_N(\theta_1, \theta_2)\*Q_N(\theta_2, \theta_3)\*Q_N(\theta_3, \theta_1)\*
d\theta_1d\theta_2 d\theta_3\\
\label{414}
&4\*\int_{\mathbb{T}^3} f(\theta_1-\theta_2)\* f(\theta_2-\theta_3)\*Q_N(\theta_1, \theta_3)\*Q_N(\theta_3, \theta_2)\*Q_N(\theta_2, \theta_1)\*
d\theta_1d\theta_2 d\theta_3.
\end{align}
Again using $\hat{f}(0)=0$ we can rewrite
(\ref{412}) as
\begin{align}
\label{415}
4\*N\*\sum_{|k|< N}(|k|-N)|\hat{f}(k)|^2.
\end{align}
The terms (\ref{413}) and (\ref{414}) are equal to each other and together contribute
\begin{align}
\label{416}
8\*\sum_{0\leq j,k,l\leq N-1}\hat{f}(j-k)\hat{f}(k-l).
\end{align}
We now turn our attention to (\ref{403}). We can rewrite it as
\begin{align}
\label{417}
&2\*\int_{\mathbb{T}^4} f(\theta_1-\theta_2)\* f(\theta_3-\theta_4)\*|Q_N(\theta_1, \theta_3)|^2\*|Q_N(\theta_2, \theta_4)|^2 
\*d\theta_1d\theta_2 d\theta_3 d\theta_4\\
\label{418}
&-2\*\int_{\mathbb{T}^4} f(\theta_1-\theta_2)\* f(\theta_3-\theta_4)\*Q_N(\theta_1, \theta_2)\*Q_N(\theta_2, \theta_3)\*Q_N(\theta_3, \theta_4)
\*Q_N(\theta_4, \theta_1)
\*d\theta_1d\theta_2 d\theta_3 d\theta_4\\
\label{419}
&-2\*\int_{\mathbb{T}^4} f(\theta_1-\theta_2)\* f(\theta_3-\theta_4)\*Q_N(\theta_1, \theta_2)\*Q_N(\theta_2, \theta_4)\*Q_N(\theta_4, \theta_3)
\*Q_N(\theta_3, \theta_1)
\*d\theta_1d\theta_2 d\theta_3 d\theta_4\\
\label{420}
&-2\*\int_{\mathbb{T}^4} f(\theta_1-\theta_2)\* f(\theta_3-\theta_4)\*Q_N(\theta_1, \theta_4)\*Q_N(\theta_4, \theta_2)\*Q_N(\theta_2, \theta_3)
\*Q_N(\theta_3, \theta_1)
\*d\theta_1d\theta_2 d\theta_3 d\theta_4.
\end{align}
The term (\ref{417}) is equal to
\begin{align}
\label{421}
&2\*\sum_{|k|<N} (N-|k|)^2\*|\hat{f}(k)|^2=\\
\label{422}
&2\*\sum_{|k|<N} |k|^2\*|\hat{f}(k)|^2 +2\*N^2\*\sum_{|k|<N} |\hat{f}(k)|^2 -4\*N\*\sum_{|k|<N} |k|\*|\hat{f}(k)|^2.
\end{align}
Terms  (\ref{418}) and (\ref{419}) are equal to each other and together contribute
\begin{align}
\label{423}
-4\*\sum_{0\leq j,k,l\leq N-1}\hat{f}(j-k)\hat{f}(k-l).
\end{align}
Finally, the expression (\ref{420}) can be rewritten as
\begin{align}
\label{424}
2\*\sum_{0\leq j,k,l,m\leq N-1}\hat{f}(j-k)\hat{f}(k-l)\chi_{(j-m=k-l)}.
\end{align}
Combining (\ref{411}), (\ref{415}-\ref{416}), and (\ref{422}-\ref{424}), we arrive at the formula (\ref{4.0}-\ref{4.3}) for the variance of $S_N(f).$

To finish the proof of Proposition 4.1 we have to carefully evaluate each of the terms in (\ref{4.1}-\ref{4.3}).
The term (\ref{4.1}) can be rewritten use the Placherel theorem 
\[ \widehat{f^2}(0)=\sum_{-\infty}^{\infty} \hat{f}(-k)\*\hat{f}(k)=\sum_{-\infty}^{\infty} |\hat{f}(k)|^2. \] 
as
        \begin{align}
            2N^2\left(\sum_{k\in\mathbb{Z}}|\hat{f}(k)|^2-\sum_{|k|\leq N-1}|\hat{f}(k)|^2\right)= 4N^2\sum_{k\geq N}|\hat{f}(k)|^2.
        \end{align}

Next, we rewrite the terms in (\ref{4.2}). We start with the first one:
        \begin{align}
    \label{4.5}
            4\sum_{0\leq j,k,l\leq N-1}\hat{f}(j-k)\hat{f}(k-l)&= 4\sum_{|s|,|t|\leq N-1} \hat{f}(s)\hat{f}(t)\max(0, N-(\max(0,s,t)-\min(0,s,t)))\\
    \label{4.6}
            &= 4\sum_{|s|,|t|\leq N-1} \hat{f}(s)\hat{f}(t)\max(0, N-L(s,t)),
        \end{align}
        where 
            \[
                L(s,t)= \begin{cases}
                            \max(|s|,|t|),& \text{if } sgn(s)=sgn(t)\\
                            |s|+|t|,              & \text{otherwise}.
                        \end{cases}
            \]
        Splitting up the sum and recalling that $\hat{f}(s)=\hat{f}(-s)$, we can further rewrite the first term in (\ref{4.2}) as
            \[
                2\sum_{\substack{|s|,|t|\leq N-1\\ |s|+|t|\leq N-1}} \hat{f}(s)\hat{f}(t)(N-(|s|+|t|)) + 
4\sum_{\substack{|s|,|t|\leq N-1\\ sgn(s)=sgn(t)}} \hat{f}(s)\hat{f}(t)(N-\max(|s|,|t|)).
            \]
        We rewrite the second term in (\ref{4.2}) as
            \[
                2\sum_{0\leq j,k,l,m\leq N-1}\hat{f}(j-k)\hat{f}(k-l)\mathlarger\chi_{(j-m=k-l)}= 
2\sum_{\substack{|s|,|t|\leq N-1\\ |s|+|t|\leq N-1}} \hat{f}(s)\hat{f}(t)(N-(|s|+|t|)).
            \]

        Thus, (\ref{4.2}) becomes
            \begin{align*}
                4\sum_{\substack{|s|,|t|\leq N-1\\ sgn(s)=sgn(t)}} \hat{f}(s)\hat{f}(t)(N-\max(|s|,|t|)),
            \end{align*}
        which can be rewritten as
            \begin{align}
            \label{4.7}
                2N\sum_{\substack{|s|,|t|\leq N-1}} \hat{f}(s)\hat{f}(t)- 
2\sum_{\substack{|s|,|t|\leq N-1\\ sgn(s)=sgn(t)}} \hat{f}(s)\hat{f}(t)(|s-t|+|s+t|).
            \end{align}
        Combining (\ref{4.7}) with the first two terms of (\ref{4.3}), we have a term of order $N$,
        \begin{align}
        \label{4.8}
            2N\left(\sum_{\substack{|s|,|t|\leq N-1}} \hat{f}(s)\hat{f}(t) - \sum_{|k|\leq N-1}\hat{f^2}(k)\right),
        \end{align}
        and a term of order constant,
            \begin{align}
        \label{4.9}
                2\left(\sum_{|k|\leq N-1}|k|\hat{f^2}(k) - 
\sum_{\substack{|s|,|t|\leq N-1\\ sgn(s)=sgn(t)}} |s-t|\hat{f}(s)\hat{f}(t)- 
\sum_{\substack{|s|,|t|\leq N-1\\ sgn(s)=sgn(t)}}|s+t|\hat{f}(s)\hat{f}(t)\right).
            \end{align}
        The expression (\ref{4.8}) can be rewritten as
            \begin{align}
            \label{4.10}
                &2N\left(\sum_{\substack{|s|,|t|\leq N-1}} \hat{f}(s)\hat{f}(t) - \sum_{|s+t|\leq N-1}\hat{f}(s)\hat{f}(t)\right)\nonumber\\
                =&2N\sum_{\substack{|s|,|t|\leq N-1\\N\leq |s+t|}} \hat{f}(s)\hat{f}(t) - 
2N\sum_{\substack{|s+t|\leq N-1\\N\leq \max(|s|,|t|) }}\hat{f}(s)\hat{f}(t).
            \end{align}
        Furthermore, (\ref{4.9}) can be rewritten as follows: 
            \begin{align*}
                2\sum_{|s+t|\leq N-1}|s+t|\hat{f}(s)\hat{f}(t) - 
2\sum_{\substack{|s|,|t|\leq N-1\\ sgn(s)=sgn(t)}} |s-t|\hat{f}(s)\hat{f}(t)- 
2\sum_{\substack{|s|,|t|\leq N-1\\ sgn(s)=sgn(t)}}|s+t|\hat{f}(s)\hat{f}(t).\\
            \end{align*}
        We break up the sum into two parts, namely
            \begin{align}
                \label{4.11}
                2\sum_{\substack{|s+t|\leq N-1\\ sgn(s)\neq sgn(t)}}|s+t|\hat{f}(s)\hat{f}(t) - 
2\sum_{\substack{|s|,|t|\leq N-1\\ sgn(s)=sgn(t)}}|s-t|\hat{f}(s)\hat{f}(t)
            \end{align}
       and    
            \begin{align}
                \label{4.12}
                2\sum_{\substack{|s+t|\leq N-1\\ sgn(s)= sgn(t)}}|s+t|\hat{f}(s)\hat{f}(t) - 
2\sum_{\substack{|s|,|t|\leq N-1\\ sgn(s)=sgn(t)}}|s+t|\hat{f}(s)\hat{f}(t)\text{\textcolor{red}{.}}
            \end{align}
        The expression (\ref{4.11}) can be rewritten as 
            \begin{align}
                &2\sum_{\substack{|s-t|\leq N-1\\ sgn(s)= sgn(t)}}|s-t|\hat{f}(s)\hat{f}(t) - 
2\sum_{\substack{|s|,|t|\leq N-1\\ sgn(s)= sgn(t)}}|s-t|\hat{f}(s)\hat{f}(t)\nonumber\\
                &=4\left(\sum_{\substack{|s-t|\leq N-1\\ 1\leq s,t}}|s-t|\hat{f}(s)\hat{f}(t) - 
\sum_{\substack{1\leq s,t \leq N-1}}|s-t|\hat{f}(s)\hat{f}(t)\right)\nonumber\\
                &= 4\left(\sum_{\substack{|s-t|\leq N-1\\ 1\leq s,t}}|s-t|\hat{f}(s)\hat{f}(t) - 
\sum_{\substack{1\leq s,t \leq N-1\\ |s-t|\leq N-1}}|s-t|\hat{f}(s)\hat{f}(t)\right)\nonumber\\
        \label{4.13}
                &= 4\left(\sum_{\substack{|s-t|\leq N-1\\ N\leq \max(s,t) \\1\leq s,t}}|s-t|\hat{f}(s)\hat{f}(t)\right)\text{\textcolor{red}{.}}
            \end{align}
        Rewriting (\ref{4.12}), we have
            \begin{align*}
                4\sum_{\substack{s+t\leq N-1\\1\leq s,t \leq N-1}}(s+t)\hat{f}(s)\hat{f}(t) - 
4\sum_{1\leq s,t\leq N-1}(s+t)\hat{f}(s)\hat{f}(t)
            \end{align*}
            \begin{align}
            \label{4.14}
                = -4\sum_{\substack{1\leq s,t\leq N-1\\ N\leq s+t}}(s+t)\hat{f}(s)\hat{f}(t).
            \end{align}
        Combining  the last term in (\ref{4.3}) with (\ref{4.5}), (\ref{4.10}), (\ref{4.13}), and (\ref{4.14}) gives
            \[
                \var(S_N(f)) = 2\sum_{|s|\leq N-1}|s|^2|\hat{f}(s)|^2
            \]
            \[
                +4N^2\sum_{N\leq  s}|\hat{f}(s)|^2
                +2N\sum_{\substack{|s|,|t|\leq N-1\\N\leq |s+t|}} \hat{f}(s)\hat{f}(t) - 
2N\sum_{\substack{|s+t|\leq N-1\\N\leq \max(|s|,|t|) }}\hat{f}(s)\hat{f}(t)
            \]
            \[  
                +4\sum_{\substack{|s-t|\leq N-1\\ N\leq \max(s,t) \\1\leq s,t}}|s-t|\hat{f}(s)\hat{f}(t)
                -4\sum_{\substack{1\leq s,t\leq N-1\\ N\leq s+t}}(s+t)\hat{f}(s)\hat{f}(t),
            \]
        which can be further rewritten as
            \[
                \var(S_N(f)) = 4\sum_{1\leq s\leq N-1}|s|^2|\hat{f}(s)|^2
            \]
            \begin{align*}
                +4(N^2-N)\sum_{N\leq  s}|\hat{f}(s)|^2
                +4N\sum_{\substack{1\leq s,t \leq N-1\\N\leq s+t}} \hat{f}(s)\hat{f}(t) - 
4N\sum_{\substack{1\leq |s-t|\leq N-1\\N\leq \max(s,t)\\ 1\leq s,t }}\hat{f}(s)\hat{f}(t)
            \end{align*}
            \[  
                +4\sum_{\substack{1\leq |s-t|\leq N-1\\ N\leq \max(s,t) \\1\leq s,t}}|s-t|\hat{f}(s)\hat{f}(t)
                -4\sum_{\substack{1\leq s,t\leq N-1\\ N\leq s+t}}(s+t)\hat{f}(s)\hat{f}(t).
            \]
        Combining like sums gives the desired result. Proposition 4.1 is proven.
        \end{proof}
        
        Now, we turn our attention to the proof of Proposition 4.3. It will follow from Proposition 4.1 and the following technical 
lemma that allows us to control the negligible terms.

        \begin{lemma} Let $f'\in L^2(\mathbb{T}).$ Then, as $N\to\infty$, we have
            \begin{enumerate}
                \item[(i)]
                    \begin{align*}
                        \sum_{\substack{1\leq s,t\leq N\\ s+t\geq N+1}}s|\hat{f}(s)|\cdot|\hat{f}(t)|\to 0,
                    \end{align*}
                \item[(ii)]
                    \begin{align*}
                        (N+1)\sum_{\substack{s-t\leq N\\s\geq N+1\\1\leq t\leq N}}|\hat{f}(s)|\cdot|\hat{f}(t)|\to 0,
                    \end{align*}    
                \item[(iii)]
                    \begin{align*}
                        N\sum_{\substack{|s-t|\leq N-1\\s,t\geq N}}|\hat{f}(s)|\cdot|\hat{f}(t)|\to 0.
                    \end{align*}
            \end{enumerate}
            
        \end{lemma}
       
We first quickly prove Proposition 4.3 modulo Lemma 4.4 and then prove Lemma 4.4.
        \begin{pfo}{\textit{Proposition 4.3}}$\\$
             Recall that $\beta=2$ and we require that $\sum_{s=1}^\infty s^2[\hat{f}(s)]^2<\infty$, i.e. 
             $f\in \mathbb{H}^1(\mathbb{T})$. We examine the last four sums on the r.h.s. of the formula for  
             $\var_N(S_N(f))$  in Proposition 4.1. Our goal is to show that these four sums go to zero as $N\to \infty.$
             The analysis of the first two sums is trivial, since
                \[
                   0\leq \sum_{s\geq N}N[\hat{f}(s)]^2\leq  \sum_{s\geq N}N^2[\hat{f}(s)]^2\leq  \sum_{s \geq N}s^2[\hat{f}(s)]^2,
                \]
           which goes to zero under our stated assumptions. The remaining two sums require a little bit more work done in Lemma 4.4. We have
                \[
                    \left|\sum_{\substack{1\leq s,t\leq N-1\\N+1\leq s+t}}((s+t)-N) \hat{f}(s)\hat{f}(t) \right| \leq 2\sum_{\substack{1\leq s,t
                    \leq N-1\\N+1\leq s+t}}(s+t)|\hat{f}(s)|\cdot|\hat{f}(t)|= 
4\sum_{\substack{1\leq s,t\leq N-1\\N+1\leq s+t}}s|\hat{f}(s)|\cdot|\hat{f}(t)|.
                \]
            It follows from Lemma 4.4(i) that the r.h.s. goes to zero as $N\to\infty$. Finally, we observe that
                \[
                    \left|\sum_{\substack{1\leq s,t \\ 1\leq |s-t|\leq N-1\\ N\leq \max(s,t)}}(N-|s-t|)\hat{f}(s)\hat{f}(t)\right| 
                    \leq 2N\sum_{\substack{1\leq s,t \\ 1\leq |s-t|\leq N-1\\ N\leq \max(s,t)}}|\hat{f}(s)|\cdot|\hat{f}(t)|
                \]
                \[
                    = 4N\sum_{\substack{|s-t|\leq N-1\\ s\geq N\\1\leq t \leq N-1}}|\hat{f}(s)|\cdot|\hat{f}(t)| + 
2N\sum_{\substack{|s-t|\leq N-1\\ s,t\geq N}}|\hat{f}(s)|\cdot|\hat{f}(t)|
                \]
            The first term goes to zero by Lemma 4.4 (ii) and the second term goes to zero by Lemma 4.4 (iii). 
This completes the proof of Proposition 4.3 modulo Lemma 4.4. 
        \end{pfo}
        \\\\
        The rest of the section is devoted to the proof of Lemma 4.4.\\\\
            \begin{pfo}{\textit{Lemma 4.4}}$\\$
                Let $x_s=s|\hat{f}(s)|$ for $1\leq s \leq N$ and $X_N=\{x_s\}_{s=1}^N$. By the assumption of Lemma 4.4 
the Euclidean norm of the vector $X_N$ is bounded in $N$. Note that
                    \begin{align}
                    \label{4.15}
                        \sum_{\substack{1\leq s,t\leq N\\ s+t\geq N+1}}s|\hat{f}(s)|\cdot|\hat{f}(t)|
                        =\sum_{t=1}^{N} x_t \cdot \left(\frac{1}{t}\sum_{s=N-t+1}^{N} x_s\right)
                        =\sum_{t=1}^N x_t \cdot \left(\frac{1}{t}\sum_{s=1}^t (U_N\*X_N)_s\right)
                        =\langle X_N,A_N X_N\rangle,
                    \end{align}
            with $A_N=B_N\*U_N$, where $U_N$ is a unitary permutation 
matrix given by $(U_N)_{s,t}= \mathds{1}_{(t=N-s+1)}$ and $B_N$ is a lower triangular matrix given by $(B_N)_{s,t}= (1/s)\mathds{1}_{(t\leq s)}$. 
In particular,
            \[
                B_N = 
                        \renewcommand\arraystretch{1.25}
                        \begin{pmatrix}
                            1           & 0           & 0           & 0  & \dots  & 0 \\
                            \frac{1}{2} & \frac{1}{2} &     0       &  0 & \dots  & 0 \\
                            \frac{1}{3} & \frac{1}{3} & \frac{1}{3} &  0 & \dots  & 0 \\
                            \vdots & \vdots & \vdots  & \vdots      & \ddots      & \vdots \\
                            \frac{1}{N} & \frac{1}{N} &\frac{1}{N}  & \frac{1}{N} &\dots  & \frac{1}{N}
                        \end{pmatrix}
            \]
            Our goal is to show that the expression in 
(\ref{4.15}) vanishes in the limit of large $N.$  First we show that the operator norm of the matrix $A_N$ is bounded in $N$. 
Indeed, $B_N(B_N)^T= B_N+(B_N)^T-D$, where $D_{s,t}=(1/s)\mathds{1}_{(s=t)}$. This gives us the bound 
$||B_N||_{op}^2\leq 2||B_N||_{op}+1$, so $||A_N||_{op}=||B_N||_{op}\leq 3$. The fact that $A_N$ weakly converges to $0$ finishes the proof of the Lemma. 
Indeed,
            \begin{align*}
                \langle X_N, A_NX_N \rangle &= \left\langle X_N-\sum_{s=1}^L \langle e_s, X_N\rangle e_s ,\hspace{3mm}A_NX_N \right\rangle + 
\left\langle\sum_{s=1}^L \langle e_s, X_N\rangle e_s ,\hspace{3mm}A_NX_N\right\rangle\\
                &=\left\langle X_N-\sum_{s=1}^L \langle e_s, X_N\rangle e_s ,\hspace{3mm}A_NX_N \right\rangle + 
\sum_{s=1}^L\left\langle \langle e_s, X_N\rangle e_s ,\hspace{3mm}A_NX_N\right\rangle\\
                &=\left\langle X_N-\sum_{s=1}^L \langle e_s, X_N\rangle e_s ,\hspace{3mm}A_NX_N \right\rangle + 
\sum_{s=1}^L x_s(A_NX_N)_s\\
                &=\left\langle X_N-\sum_{s=1}^L \langle e_s, X_N\rangle e_s ,\hspace{3mm}A_NX_N \right\rangle + 
\sum_{s=1}^L x_s\frac{x_{N-s+1}+\dots +x_{N}}{s}.\\
            \end{align*}
            Let $\epsilon>0$. Then we can choose $L$ sufficiently large such that, 
              \begin{align*}
                  |\langle X_N, A_NX_N \rangle| 
                  &\leq \left|\left\langle X_N-\sum_{s=1}^L \langle e_s, X\rangle e_s ,\hspace{3mm}A_NX_N \right\rangle\right| + 
\left|\sum_{s=1}^L x_s\frac{x_{N-s+1}+\dots +x_{N}}{s}\right|\\
                  &\leq \epsilon + \left|\sum_{s=1}^L x_s\frac{x_{N-s+1}+\dots +x_{N}}{s}\right|\\
                  &\to \epsilon
              \end{align*}
            Since this holds for arbitrary $\epsilon$, we can conclude that $\langle X_N, A_N\*X_N\rangle\to 0$. This completes the proof of Lemma 4.4(i).\\
            
To prove part (ii), let $B_N$ be defined as in the proof of part $(i)$. Similarly, 
let $x_s=s|\hat{f}(s)|$ and $X_N=\{x_s\}_{s=1}^{2N}$. Now, $X_N$ is a $2N$-dimensional vector bounded, uniformly with respect to $N,$  
in  Euclidean norm. Observe that
                \begin{align*}
                    N\sum_{\substack{s-t\leq N\\s\geq N+1\\1\leq t\leq N}}|\hat{f}(s)|\cdot|\hat{f}(t)|
                    &\leq \sum_{t=1}^N x_t\left(\frac{1}{t}\sum_{s=N+1}^{N+t}x_s\right)\\
                    &= \langle C_N X_N, M_N X_N\rangle,
                \end{align*}
            where 
                 \[
                    C_N =\renewcommand\arraystretch{1.25}
                        \begin{pmatrix}
                           I_N & 0\\
                           0   & 0
                        \end{pmatrix}
                        \hspace{5mm}\text{and}\hspace{5mm}
                    M_N = 
                        \renewcommand\arraystretch{1.25}
                        \begin{pmatrix}
                           B_N & 0\\
                           0 & 0
                        \end{pmatrix}
                        \begin{pmatrix}
                           0 & I_N\\
                           I_N & 0
                        \end{pmatrix}.
                 \]
            Using the same arguments as in the proof of (i), we can see that $||M_N||_{op}\leq 3$.  Clearly, $||C_N||_{op}=1.$
The rest of the proof is similar to that of (i). Indeed, for any $\epsilon>0$, we can choose $L$ sufficiently large such that
                 \begin{align*}
                     \left|\langle C_N X_N, M_N X_N\rangle\right| 
                     &\leq \left|\left\langle C_N X_N - \sum_{k=1}^L\langle e_k, C_N X_N\rangle e_k, M_N X_N\right\rangle\right| + 
\left|\left\langle \sum_{k=1}^L\langle e_k, C_N X_N\rangle e_k, M_N X_N\right\rangle\right|\\
                     &\leq \epsilon + \left|\sum_{k=1}^Lx_k(M_N X_N)_k\right|\\
                     &= \epsilon + \left|\sum_{k=1}^Lx_k\frac{x_{N+1}+\dots + x_{N+k}}{k}\right|\\
                     &\to \epsilon
                 \end{align*}
            In the above inequalities, we assume $N$ is large enough such that we can choose $L\leq N$. This completes the proof of (ii).\\
            
            To prove (iii), we start by splitting up the sum into two parts, namely
                \begin{align}
\label{myt}
                 N\sum_{\substack{t-N+1
                 \leq s\leq N+t-1\\t\geq N}}|\hat{f}(s)|\cdot|\hat{f}(t)| - N\sum_{\substack{t-N+1\leq s \leq N-1\\t\geq N}}|\hat{f}(s)|\cdot|\hat{f}(t)|.
               \end{align}
                The second sum in (\ref{myt}) goes to zero by (ii). Let $x_s=s|\hat{f}(s)|$ for $s\geq 1$. 
Then $X=\{x_s\}_{s=1}^\infty\in \ell^2(\mathbb{N})$. We can bound the first sum in (\ref{myt}) as follows: 
                    \begin{align*}
                        N\sum_{\substack{N-t+1\leq s\leq N+t-1\\t\geq N}}|\hat{f}(s)|\cdot|\hat{f}(t)| 
                        &\leq \sum_{t=N}^\infty x_t\left(\frac{1}{t}\sum_{s=t-N+1}^{N+t-1}x_s\right)\\
                        &= \langle L^{N-1} X, R_N X \rangle,
                    \end{align*}
                where $L,R_N$ are bounded linear operators on $\ell^2(\mathbb{N})$. In particular, $L,R_N$ are 
infinite dimensional matrices such that $L_{s,t}=\mathds{1}_{t=s+1}$ and  $(R_N)_{s,t}=\frac{1}{N+s-1}\mathds{1}_{(s\leq t \leq s+2N-2)}$.
                \[
                    L= \begin{pmatrix}
                        0 & 1 & 0 & 0 &\cdots \\
                        0&  0 & 1 & 0 &\cdots \\
                        \vdots&  \vdots&\ddots& \ddots&\ddots\\
                    \end{pmatrix}
                \hspace{5mm}
                    R_N=\begin{pmatrix}
                        \frac{1}{N} & \frac{1}{N} & \frac{1}{N} & \cdots & \frac{1}{N} & 0 & 0& 0&\cdots \\
                        0&  \frac{1}{N+1}& \frac{1}{N+1}& \frac{1}{N+1} &\cdots  & \frac{1}{N+1} & 0 & 0& \cdots \\
                        0&  0   & \frac{1}{N+2}& \frac{1}{N+2} &\frac{1}{N+2}  & \cdots & \frac{1}{N+2} & 0 &\cdots \\
                        \vdots& \vdots& \ddots &\ddots&\ddots &\ddots&\ddots&\ddots&\ddots
                    \end{pmatrix}
                \]
            Clearly $||L||_{op}=1$ and
                \[
                    ||R_N||_{op}\leq ||R_N||_2 = \sqrt{(2N-1)\sum_{k=N}^\infty \frac{1}{k^2}}\leq \sqrt{\frac{2N-1}{N-1}}\leq \sqrt{3}
                \]
            for $N\geq 2$. Now, by the Cauchy-Schwarz inequality,
                \begin{align*}
                    |\langle L^{N-1} X, R_N X \rangle|^2&\leq ||L^{N-1}X||^2_2\cdot||R_N||^2_{op} \cdot || X||^2_2 \\
                    &\leq 3\left(\sum_{k=N}^\infty |k|^2 |\hat{f}(k)|^2\right)\left(\sum_{k=1}^\infty |k|^2 |\hat{f}(k)|^2\right)\to 0 .
                \end{align*}
    
            This completes the proof of Lemma 4.4.
            \end{pfo}

\section{Joint Cumulants of Linear Statistics for $\beta=2$}

The goal of this section is to study joint cumulants of CUE linear statistics, i.e. 
of random variables $T_{N}(f_j):=\sum_{i=1}^{N}f_{j}(\theta_{i}),\ \ j\geq 1,$
where $\{\theta_i: 1\leq i \leq N\}$ are the eigenvalues of an $N\times N$ CUE matrix. We refer the reader for definition of joint cumulants 
to \cite{malyshev}. Recall that for a family of random variables $\{X_{\alpha \in \mathcal{A}} \},$

\begin{align}
\label{cummulants}
\kappa_{i_1,\ldots,i_n}:=\kappa(X_{i_1}, \ldots, X_{i_n})= \sum_{\pi} (|\pi|-1)!(-1)^{|\pi|-1} \prod_{B\in \pi} \E \left( \prod_{i \in B} X_i \right),
\end{align}
where the sum is over all partitions $\pi$ of $\{i_1, \ldots, i_n\},$ $B$ runs through the list of all blocks of the partition $\pi,$ and $|\pi|$ is 
the number of blocks in the partition. Joint cumulants are symmetric, i.e.
\[ 
\kappa(X_1, \ldots, X_n)=\kappa(X_{\sigma(1)}, \ldots, X_{\sigma(n)}), \ \ \sigma \in S_n,
\]
and have the multilinearity property.
\[
\kappa(c_1\*Y_1 +c_2\*Y_2, X_2, \ldots, X_n)=c_1\*\kappa(Y_1, X_2, \ldots, X_n) +c_2\*\kappa(Y_2, X_2, \ldots, X_n).
\]
The joint moments are expressed in terms of joint cumulants as
\begin{align}
\label{moments}
\E(X_1, \ldots, X_n):=\E \prod_{1\leq i\leq n} X_i= \sum_{\pi} \prod_{B\in \pi} \kappa(X_i: i \in B).
\end{align}
The joint cumulant of two random variables is the covariance. 
Finally, we note that as the joint moments are expressed in terms of the 
partial derivatives at the origin of the Laplace transform (exponential moment) 
$g(\lambda_1, \ldots, \lambda_n)=\E e^{\lambda_1\*X_1+\ldots \lambda_n\*X_N},$ the joint cumulants can be expressed in terms of the partial derivatives 
at the origin of $\log(g(\lambda_1, \ldots, \lambda_n)).$

Denote by $t_{N,k}$ the trace of the $k$-th power of a CUE matrix, i.e.
\begin{align}
\label{kkk}
t_{N,k}:=\sum_{j=1}^N e^{i\*k\*\theta_j}, \ \ k=0,\pm 1, \pm 2, \ldots.
\end{align}

In the next lemma we study joint cumulants of the traces of powers of a CUE matrix
\begin{align}
\label{kkkk}
\kappa_n^{(N)}(k_1, \ldots, k_n):=\kappa(t_{N,k_1}, \ldots, t_{N,k_n}).
\end{align}

The following result follows from the formula (2.8) of \cite{sasha} and the fact that $\kappa_n^{(N)}(k_1, \ldots, k_n)$ is a symmetric 
function:

\begin{lemma}
Let $n>1.$ Then
\begin{align}
& & \kappa_n^{(N)}(k_1, \ldots, k_n)=\sum_{m=1}^n \frac{(-1)^{m-1}}{m} \*\sum_{\substack{(n_1, \ldots, n_m):
\\ n_1+\ldots n_m=n, \ n_1, \ldots n_m\geq 1}} 
\frac{1}{n_1!\cdots n_m!} \* \label{semi} \\
& & \sum_{\sigma \in S_n} \* \#\left\{u: 0\leq u\leq N-1, 0\leq u+ \sum_{i=1}^{n_1} k_{\sigma(i)} \leq N-1, \ldots, 
0\leq u+ \sum_{i=1}^{n_1+\ldots+n_{m-1}} k_{\sigma(i)} \leq N-1 \right\} \nonumber
\end{align}
for $k_1+\ldots +k_n=0$ and equals zero otherwise.

For $n=1$ one has $\kappa_1^{(N)}(k)=N$ for $k=0$ and $\kappa_1^{(N)}(k)=0$ otherwise.
\end{lemma}

Denote
\begin{align}
\label{I}
& I_N(n_1, \ldots, n_m; k_1,\ldots, k_n)
:=\\
& \#\left\{u: 0\leq u\leq N-1, 0\leq u+ \sum_{i=1}^{n_1} k_i\leq N-1, \ldots, 
0\leq u+ \sum_{i=1}^{n_1+\ldots+n_{m-1}} k_i \leq N-1 \right\}, \nonumber
\end{align}
where $n_1+\ldots+n_m=n $  and $n_1, \ldots, n_m\geq 1.$
Clearly,
\begin{align}
\label{II} & I_N(n_1, \ldots, n_m; k_1, \ldots, k_n)= \\
& \max\left(0, N-\max\left(0, \sum_i^{n_1} k_i, \sum_i^{n_1+n_2} k_i, \ldots, \sum_i^{n_1+\ldots+n_{m-1}} k_i\right)-
\max\left(0, \sum_i^{n_1} (-k_i), \ldots, \sum_i^{n_1+\ldots+n_{m-1}} (-k_i)\right)\right).\nonumber
\end{align}

Thus, 
\begin{align}
\label{III}
0\leq I_N(n_1, \ldots, n_m; k_1, \ldots, k_n) \leq N.
\end{align}
Moreover, if $\sum_1^n |k_i| \leq N, $ then
\begin{align}
\label{IIII} & I_N(n_1, \ldots, n_m; k_1, \ldots, k_n)= \\
& N-\max\left(0, \sum_i^{n_1} k_i, \sum_i^{n_1+n_2} k_i, \ldots, \sum_i^{n_1+\ldots+n_{m-1}} k_i\right)-
\max\left(0, \sum_i^{n_1} (-k_i), \ldots, \sum_i^{n_1+\ldots+n_{m-1}} (-k_i)\right).\nonumber
\end{align}
Next result follows from the above Lemma 5,1 (\ref{I}-\ref{IIII}), and the combinatorial Lemma 2 from \cite{sasha}
(for the convenience of the reader, we formulate 
the combinatorial Lemma 2 from \cite{sasha} in Appendix 2 as Lemma 9.1.)

\begin{lemma}

(i) $|\kappa_n^{(N)}(k_1, \ldots, k_n)|\leq const_n \*N,$
where $const_n$ is some universal constant that depends only on $n.$

(ii) Let $n\geq 1, $ and $\sum_1^n k_i\neq 0.$ Then $\kappa_n^{(N)}(k_1, \ldots, k_n)=0.$

(iii)  Let $\sum_1^n k_i=0, \ \ \sum_1^n |k_i| \leq N,$ and $n>2.$ Then $\kappa_n^{(N)}(k_1, \ldots, k_n)=0.$

(iv) Let $n=2$ and $k_1=-k_2.$ Then $\kappa_2^{(N)}(k_1, k_2)=\kappa_2^{(N)}(k_1, -k_1)=\min(N, |k_1|).$
\end{lemma}
\begin{proof}

(i) follows from Lemma 5,1 and (\ref{I}-\ref{IIII}).

(ii) follows form Lemma 5.1.

(iii) and (iv) follow from Lemma 5.1, (\ref{IIII}) and the combinatorial lemma from \cite{sasha} (see Lemma 9.1 in Appendix 2).

\end{proof}

\section{Proof of Theorem 2.4}
This section is devoted to the proof of Theorem 2.4. We use Lemma 5.2 and the Lindeberg-Feller condition when $\beta=2$ and Jiang-Matsumoto moment
estimates \cite{jm} for arbitrary $\beta.$

\begin{pfo}{\textit{Theorem 2.4.}}$\\$
        When $N$ is sufficiently large, the support of $f(L_N\cdot)$ is contained in the interval $[-\pi,\pi]$. In particular, 
$f(L_N\cdot)$ has a Fourier Series given by
            \[
                f(L_N\theta)=\sum_{k\in\mathbb{Z}}\frac{1}{\sqrt{2\pi}L_N} \hat{f}\left(\frac{k}{L_N}\right)e^{ik\theta}
            \]
        where $\theta\in[-\pi,\pi)$  and 
            \[
                \hat{f}(\xi)=\frac{1}{\sqrt{2\pi}}\int_\mathbb{R}f(x)e^{-i\xi x}dx
            \]
        is the Fourier transform of $f.$  
        Extend $f(L_N\cdot)$ $2\*\pi$-periodically to the whole real line. 
Then
            \[
                S_N(f(L_N\cdot)) = \sum_{1\leq j\neq k\leq N} f(L_N(\theta_j-\theta_k)_c) = 
\sum_{k\in \mathbb{Z}} \frac{1}{\sqrt{2\pi}L_N}\hat{f}\left(\frac{k}{L_N}\right)\left(\left|\sum_{m=1}^N e^{ik\theta_m}\right|^2-N\right).
            \]

Consider first the case $\beta=2, $ so $\{\theta_j\}_{j=1}^N$ are distributed according to $CUE(N)$. We have
            \begin{align}
            \label{6.1}
                S_N(f(L_N\cdot))-\E(S_N(f(L_N\cdot))) = 2\sum_{k\geq 1} 
\frac{1}{\sqrt{2\pi}L_N}\hat{f}\left(\frac{k}{L_N}\right)\left(\left|\sum_{m=1}^N e^{ik\theta_m}\right|^2-\min(k, N)\right),
            \end{align}
where we used the fact that
$ \E \left|\sum_{m=1}^N e^{ik\theta_m}\right|^2= \min(k, N)$
and the Fourier series expansion for $f(L_N\cdot)$ given above.
One can directly study the variance of   $S_N(f(L_N\cdot)) $ by using the result of Proposition 4.1 and  
inserting the Fourier coefficients for the mesoscopic case into the variance formula. 
The first (leading) term in the variance formula then becomes
            \[
                L_N\left[\frac{2}{\pi L_N}\sum_{1\leq k\leq N-1}\left(\frac{k}{L_N}\right)^2\left(\hat{f}\left(\frac{k}{L_N}\right)\right)^2\right].
            \]
        The term in the square brackets is a Riemann sum which converges to $||f'||^2_2/\pi.$ Thus, the  variance of  $S_N(f(L_N\cdot))$ is of order
$L_N.$  We then normalize (\ref{6.1}) 
by $\sqrt{L_N/(2\pi)}$ and break it up into two pieces:  
            \begin{align}
            \label{6.2}
                \frac{2}{\sqrt{L_N}}\sum_{k=1}^{\sqrt{N L_N}} \frac{k}{L_N}&\hat{f}\left(\frac{k}{L_N}\right)(\varphi_k^{(N)}-1)\\ 
                &+  \frac{2}{\sqrt{L_N}}\sum_{k=\sqrt{NL_N} +1}^\infty \frac{1}{L_N}\hat{f}\left(\frac{k}{L_N}\right)
\left(\left|\sum_{m=1}^N e^{ik\theta_m}\right|^2-\min(k, N)\right),\nonumber
            \end{align}
        where 
            \[
                \varphi_k^{(N)}: =  \frac{1}{k}\left|\sum_{m=1}^N e^{ik\theta_m}\right|^2=\frac{1}{k}\*|t_{N,k}|^2.
            \]
We show in Appendix 3 that  the variance of the second sum in (\ref{6.2}) converges to zero by applying Proposition 4.1 and 
analogous arguments from the proof of Proposition 4.3.
Therefore, it is enough to study the asymptotic distribution of the first sum:
            \begin{align}
                \Sigma_N=\frac{2}{\sqrt{L_N}}\sum_{k=1}^{\sqrt{NL_N}} \frac{k}{L_N}\hat{f}\left(\frac{k}{L_N}\right)(\varphi_k^{(N)}- \E\varphi_k^{(N)})=
\frac{2}{\sqrt{L_N}}\sum_{k=1}^{\sqrt{NL_N}} \frac{k}{L_N}\hat{f}\left(\frac{k}{L_N}\right)(\varphi_k^{(N)}-1).
            \end{align}
Consider the sequence of random variables $ \left(\varphi_1^{(N)}, \varphi_2^{(N)}, \ldots, \varphi_k^{(N)},\ldots \right )_N.$ 
As $N\to \infty,$ this sequence converges, 
in finite-dimensional distributions, to a sequence of i.i.d. exponential random variables 
$ \left(\varphi_1, \varphi_2, \ldots, \varphi_k,\ldots \right ).$ 
Moreover, Lemma 5.2, specifically (ii)-(iii), implies that for any fixed $n$ and sufficiently large $N$ (depending on $n$)
all joint moments up to order $n$ of random variables $\{\varphi_k^{(N)}\}_{k=1}^{\sqrt{NL_N}}$ coincide with the corresponding 
joint moments of i.i.d. exponential random variables $\{\varphi_k\}_{k=1}^{\sqrt{NL_N}}.$ Therefore, it is enough to study the asymptotic distribution of

\begin{align}
\label{ofis}
                \Sigma_N=\frac{2}{\sqrt{L_N}}\sum_{k=1}^{\sqrt{NL_N}} \frac{k}{L_N}\hat{f}\left(\frac{k}{L_N}\right)(\varphi_k-1),
\end{align}
where we recall that $\{\varphi_k\}$ are i.i.d. exponential random variables.

This can done by routine computation. For example, one can explicitly compute the exponential moment of (\ref{ofis}) and study its asymptotics in the limit
of large $N$ showing that the exponential moment converges to that of a centered Gaussian random variable with the prescribed variance.
Below, for completeness, we show that the sequence of random variables in the above sum satisfy the Lindeberg-Feller condition 
\cite{dur}.  
         
         Let
            \[
                c_{N,k}=\frac{2k}{L_N^{(3/2)}}\hat{f}\left(\frac{k}{L_N}\right),\hspace{5mm}X_{N,k}= c_{N,k}(\varphi_k-1).
            \]
        Then $\E(X_k)=0$, $\var(X_k)=c_{N,k}^2$, and
            \begin{align*}
                \Sigma_N=  \sum_{k=1}^{\sqrt{NL_N}}X_{N,k},
            \end{align*}
        Denote by $s_N^2$ the variance of $\Sigma_N$, i.e.
            \[
                s_N^2=\sum_{k=1}^{\sqrt{NL_N}}c_{N,k}^2.
            \]
        To see that the sequence of random variables $(X_k)$ satisfy the Lindeberg-Feller condition, we check that, given $\epsilon>0$,
            \[
                \frac{1}{s_N^2} \sum_{k=1}^{\sqrt{NL_N}} \E(X_k^2 1_{|X_k|>\epsilon\sigma_N})\to 0.
            \]
        If $|c_{N,k}|=0$ for some $k$, then $\E(X_k^2 1_{|X_k|>\epsilon\sigma_N})=0$, so, without loss of generality, we will assume that 
$|c_{N,k}|>0$ for all $k$ and $N$. By direct computation, we see that 
            \begin{align*}
                \E(X_k^2 1_{|X_k|>\epsilon\sigma_N})
                &=c_{N,k}^2\E([\varphi_k^2-2\varphi_k+1] 1_{|\varphi_k-1|>\epsilon\sigma_N/|c_{N,k}|})\\
                &= c_{N,k}^2\int_{x>1+(\epsilon\sigma_N/|c_{N,k}|)}(x^2-2x+1)e^{-x}dx.
            \end{align*}      
        Now $s_N=O(1)$ and, since $f'$ is continuous and bounded, we have $1/|c_{N,K}|\geq C\sqrt{L_N}$ for some positive constant C that is 
independent of $k$ and $N$. It follows that, for large enough $N$, we can write
            \begin{align*}
                \E(X_k^2 1_{|X_k|>\epsilon\sigma_N})&\leq  c_{N,k}^2\int_{x>\gamma_N}(x^2-2x+1)e^{-x}dx\\
                &=   c_{N,k}^2 e^{-\gamma_N}(\gamma_N^2+1),
            \end{align*}
        where $\gamma_N= C\*\sqrt{L_N}\*[\epsilon\*\sigma_N]$. Clearly $\gamma_N= O(\sqrt{L_N})$ and $e^{-\gamma_N}(\gamma_N^2+1)$ goes to zero 
independent of $k$. This immediately implies 
            \[
                \frac{1}{\sigma_N^2} \sum_{k=1}^{\sqrt{NL_N}} \E(X_k^2 1_{|X_k|>\epsilon\sigma_N})\leq 
\frac{1}{\sigma_N^2}\sum_{k=1}^{\sqrt{NL_N}}c_{N,k}^2 \cdot o_N(1)= o_N(1),
            \]
        so the Lindeberg-Feller condition is satisfied and we can conclude that 
            \begin{align}
                    \frac{\Sigma_N}{s_N} \xrightarrow[]{\hspace{1mm}\mathcal{D}\hspace{2mm}} \mathcal{N}\left(0 , 1\right),
            \end{align}
        where $s_N^2$ is a Riemann sum that converges to
            \[
                2\int_\mathbb{R}[x\hat{f}(x)]^2\hspace{1mm}dx
            \]
        as $N\to\infty$. This completes the proof of Theorem 2.4 when $\beta=2$.
                    \end{pfo}

The proof in the case $\beta\neq 2$ relies on the results by Jiang and Matsumoto \cite{jm} that, in particular, state that for any finitely many 
positive integers $k_1, k_2, \ldots k_n, \ \ k_i <<N, \ 1\leq i \leq n,$
one has 
\begin{align}
\label{JM}
\E \prod_{i=1}^n \varphi_{k_i}^{(N)}= \left(\E \prod_{i=1}^n \varphi_{k_i} \right)\*\left(1+O\left(\frac{k_1+\ldots k_n}{N}\right)\right).
\end{align}

Namely, we proceed as follows. As in the case $\beta=2$ we write

\begin{align*}
&S_N(f(L_N\cdot))-\E(S_N(f(L_N\cdot))) = 2\sum_{k\geq 1} 
\frac{1}{\sqrt{2\pi}L_N}\hat{f}\left(\frac{k}{L_N}\right)
\left(\left|\sum_{m=1}^N e^{ik\theta_m}\right|^2- \E \left|\sum_{m=1}^N e^{ik\theta_m}\right|^2 \right)=\\
&\frac{2}{\sqrt{L_N}}\sum_{k=1}^{\infty} \frac{k}{L_N}\hat{f}\left(\frac{k}{L_N}\right)(\varphi_k^{(N)}- \E\varphi_k^{(N)}).
\end{align*}
We then split the last sum into three subsums, namely 

\begin{align*}
& \frac{2}{\sqrt{L_N}}\sum_{k=1}^{L_N^2} \frac{k}{L_N}\hat{f}\left(\frac{k}{L_N}\right)(\varphi_k^{(N)}- \E\varphi_k^{(N)})+
\frac{2}{\sqrt{L_N}}\sum_{k=L_N^2}^{N/10} \frac{k}{L_N}\hat{f}\left(\frac{k}{L_N}\right)(\varphi_k^{(N)}- \E\varphi_k^{(N)})\\
&+\frac{2}{\sqrt{L_N}}\sum_{k>N/10}^{\infty} \frac{k}{L_N}\hat{f}\left(\frac{k}{L_N}\right)(\varphi_k^{(N)}- \E\varphi_k^{(N)})
\end{align*}
and deal with each subsum separately. The variance of the second sum goes to zero as $N\to \infty$ since the Fourier transform of $f$ decays sufficiently  
fast for $f \in C^{\infty}_c(\R)$ and $\E |\varphi_k^{(N)}- \E\varphi_k^{(N)}|^2$ is bounded for $k\leq N/10.$  Here, the bound on the variance of 
$\varphi_k^{(N)}$ follows from (\ref{JM}).

The variance of the third subsum goes to 
zero as well. Indeed, we bound $\E |\varphi_k^{(N)}- \E\varphi_k^{(N)}|^2$ from above by $N^4/k^2$ for $k>N/10$ and again 
use a fast decay of $\hat{f}\left(\frac{k}{L_N}\right)$ to finish the argument. 

Now we turn our attention to the first subsum
\begin{align}
\label{firstsubsum}
\frac{2}{\sqrt{L_N}}\sum_{k=1}^{L_N^2} \frac{k}{L_N}\hat{f}\left(\frac{k}{L_N}\right)(\varphi_k^{(N)}- \E\varphi_k^{(N)}).
\end{align}
It follows from (\ref{JM}) that for any positive integer $l$ the $l$-th moment of (\ref{firstsubsum}) equals to the $l$-th moment
\begin{align}
\label{refref}
\frac{2}{\sqrt{L_N}}\sum_{k=1}^{\sqrt{L_N\*N}} \frac{k}{L_N}\hat{f}\left(\frac{k}{L_N}\right)(\varphi_k-1)
\end{align}
up to a vanishing error term of order $O(L_N^{l/2 +2}\*N^{-1}).$
Again, the exponential moment of (\ref{refref}) converges to that of a Gaussian random variable.
Theorem 2.4 is proven.

\section{Proof of Theorem 2.5}
The section is devoted to the proof of Theorem 2.5. The proof uses the method of moments and is combinatorial in nature.
Recall that
  \[
                S_N(f) = \sum_{1\leq j\neq k\leq N} f(N\*(\theta_j-\theta_k)_c) = 
\sum_{k\in \mathbb{Z}} \frac{1}{\sqrt{2\pi}\*N}\hat{f}(k/N)\left(\left|\sum_{m=1}^N e^{ik\theta_m}\right|^2-N\right),
            \]
where $\{\theta_1, \ldots, \theta_N\}$ are distributed according to the CUE statistics ($\beta=2$.)
To simplify the notations, we will write $S_N(f)$ for $S_N(f(N\cdot))$ for the rest of this section.
One has
\[\E S_N(f)= \sum_{k\in \mathbb{Z}} \frac{1}{\sqrt{2\pi}}\hat{f}(k/N) \*\min\left(\frac{|k|}{N}, 1\right) + \hat{f}(0)\*N^2  -f(0)\*N
\]
and
\begin{align}
            \label{7.1}
                S_N(f)-\E S_N(f) &= 2\sum_{k\geq 1} 
\frac{1}{\sqrt{2\pi}N}\hat{f}(k/N)\left(\left|\sum_{m=1}^N e^{ik\theta_m}\right|^2-\min(k, N)\right)\\
& = 2\sum_{k\geq 1} 
\frac{1}{\sqrt{2\pi}}\min \left(\frac{k}{N}, 1\right) \hat{f}(k/N)\left(\varphi_k^{(N)}-1\right), \nonumber
            \end{align}
where $\varphi_k^{(N)}=\frac{1}{\min(k,N)}\*\left|\sum_{m=1}^N e^{ik\theta_m}\right|^2.$
For $l\geq 1 $ one has

\begin{align}
            \label{7.2}
\E (S_N(f)-\E S_N(f))^l=(2/\pi)^{l/2}\*N^{-l}\*\sum_{k_1\geq 1}\ldots \sum_{k_l\geq 1} 
\E \prod_{i=1}^l  
\hat{f}(k_i/N) \* \left(t_{N, k_i}\*t_{N, -k_i}-\E t_{N, k_i}\*t_{N, -k_i}\right),
\end{align}
where we recall that the traces of powers of a CUE matrix $t_{N,k}$ are defined in (\ref{kkk}).
The mathematical expectation on the r.h.s. of (\ref{7.2}) can be written in terms of joint cumulants (\ref{kkkk}) 
using Lemma 9.2 from Appendix 2. 
Namely, the lemma states that for centered random variables
$ X_1, \ldots, X_{2n}$ with finite moments,
\begin{align}
\label{cenmoments1}
\E \prod_{1\leq i\leq l} (X_{2i-1}\*X_{2i}-\E X_{2i-1}\*X_{2i})   = \sum^*_{\pi} \prod_{B\in \pi} \kappa(X_i: i \in B),
\end{align}
where the sum on the r.h.s. of (\ref{cenmoments}) is over all partitions $\pi$ of $\{1, \ldots, 2l\}$ that do not contain singletons and 
two-element subsets of the 
form $\{2i-1, 2i\}, \ i=1, \ldots,l.$ 
In our analysis, it will be useful to identify the set $\{1, \ldots, 2l\}$ with 
the set $\{k_1, -k_1, k_2, -k_2, \ldots,k_l, -k_l\}.$

We are going to use Lemma 5.2 to evaluate the asymptotics of (\ref{7.2}).
Let us first consider the cases $l=2$ and $l=3.$ 

We start with the already established case $l=2.$ It follows from (\ref{7.2}) and (\ref{cenmoments1}) that
\begin{align}
            \label{7.3}
\E (S_N(f)-\E S_N(f))^2& =(2/\pi)\*N^{-2}\*\sum_{k_1\geq 1}\sum_{k_2\geq 1} 
\hat{f}(k_1/N) \*\hat{f}(k_2/N) \* \kappa_4^{(N)}(k_1, -k_1, k_2, -k_2) +\\
&(2/\pi)\*N^{-2}\*\sum_{k_1\geq 1}\sum_{k_2\geq 1} 
\hat{f}(k_1/N) \*\hat{f}(k_2/N) \*\kappa_2^{(N)}(k_1, -k_2) \*\kappa_2^{(N)}(-k_1, k_2) \nonumber \\
&(2/\pi)\*N^{-2}\*\sum_{k_1\geq 1}\sum_{k_2\geq 1} 
\hat{f}(k_1/N) \*\hat{f}(k_2/N) \*\kappa_2^{(N)}(k_1, k_2) \*\kappa_2^{(N)}(-k_1, -k_2). \nonumber
\end{align}
Applying Lemma 5.2 part (ii), we conclude that the third sum on the r.h.s. of (\ref{7.3}) vanishes and the only non-zero terms in the second sum 
correspond to $k_1=k_2$ in which case $\kappa_2^{(N)}(k_1, -k_2)=\kappa_2^{(N)}(-k_1, k_2)=\min(N, k_1).$ Therefore, up to a factor $N,$ 
the second sum is just a Riemann sum of the integral 
\[\frac{1}{\pi}\*\int_{\mathbb{R}} |\hat{f}(t)|^2 \*\min(|t|,1)^2\* dt.\]
Now we turn our attention to the first sum. The terms appearing in 
$\kappa_4^{(N)}(k_1, -k_1, k_2, -k_2)$ have been studied in detail in Section 4. It follows that the 
first sum is also proportional to $N,$ and the coefficient in front of $N$
is recognized as a Riemann sum of 
\[ -\frac{1}{\pi}\* \int_{|s-t|\leq 1, |s|\vee|t|\geq 1} \hat{f}(t)\*\hat{f}(s)\*
(1-|s-t|)\* ds\*dt
-\frac{1}{\pi}\* \int_{0\leq s,t\leq 1, s+t>1} \hat{f}(s)\*\hat{f}(t)\*(s+t-1) \*ds\*dt.\]
Combining these two  results together,
we obtain the variance asymptotics (\ref{vvvv}) for
the normalized random variable $(S_N(f)-\E S_N(f))\*N^{-1/2}.$

Consider now $l=3.$ Again, (\ref{7.2}), (\ref{cenmoments1}), and Lemma 5.2 give us

\begin{align}
            \label{7.33}
&\E (S_N(f)-\E S_N(f))^3 =\\
&(2/\pi)^{3/2}\*N^{-3}\*\sum_{k_1\geq 1}\sum_{k_2\geq 1} \sum_{k_3\geq 1}
\hat{f}(k_1/N) \*\hat{f}(k_2/N) \*\hat{f}(k_3/N)\* \kappa_6^{(N)}(k_1, -k_1, k_2, -k_2, k_3, -k_3) +\nonumber \\
&3\*(2/\pi)^{3/2}\*N^{-3}\*\sum_{k_1\geq 1}\sum_{k_2\geq 1} \sum_{k_3\geq 1}
\hat{f}(k_1/N) \*\hat{f}(k_2/N) \*\hat{f}(k_3/N)\*\kappa_3^{(N)}(k_1, k_2, -k_3) \*\kappa_3^{(N)}(-k_1, -k_2, k_3) \nonumber \\
&2\*(2/\pi)^{3/2}\*N^{-3}\*\sum_{k_1\geq 1}\sum_{k_2\geq 1}  \sum_{k_3\geq 1}
\hat{f}(k_1/N) \*\hat{f}(k_2/N)  \*\hat{f}(k_3/N)\*\kappa_2^{(N)}(k_1, -k_2) \*\kappa_2^{(N)}(k_2, -k_3) \* \kappa_2^{(N)}(k_3, -k_1). \nonumber 
\end{align}

It follows from Lemma 5.2 (i) that the first sum on the r.h.s. of (\ref{7.33}) is of order $N.$
It further follows from Lemma 5.2 (i) and (ii) that the second sum is restricted to $k_1+k_2=k_3$ and is also of order $N.$
Finally, the third sum is restricted to $k_1=k_2=k_3$ and is again of order $N.$  Thus, the third moment of the normalized random variable
$(S_N(f)-\E S_N(f))\*N^{-1/2}$ goes to zero in the limit $N\to \infty.$

Now consider the case $l>3.$ Below we restrict our attention to the even case $l=2\*n.$ The odd case 
$l=2\*n+1$ can be treated in a similar way. The starting point is again 
formula (\ref{7.2}). Applying (\ref{cenmoments1}) to the mathematical expectation $\E \prod_{j=1}^{2n}\left(t_{N, k_j}\*t_{N, -k_j}-
\E t_{N, k_j}\*t_{N, -k_j}\right)$ and writing the expectation as the sum of products of joint cumulants,
we split the sum into subsums labeled by the partitions $\pi$ of $\{1,2,\ldots, 4\*n\}$ with no atoms and no two-point subsets of the form
$\{2\*i-1, 2\*i\}, \ \ i=1,\ldots, 2n.$ We will denote a subsum in (\ref{7.2}) corresponding to a partition $\pi$ by $\Sigma_{\pi}.$
We make the following definition.

\begin{definition}
We call a partition of the set $\{1,2,\ldots, 4\*n\}$ (which can be also identified with the set $\{k_1, -k_1, k_2, -k_2, \ldots, k_{2\*n}, -k_{2\*n}\}$)
optimal if $\pi$ consists 
only of paired two-point blocks $\{2\*i-1, 2\*j\}, \ \{2\*i, 2\*j-1\}, \ \ 1\leq i<j\leq 2\*n,$
(so that if $\{2\*i-1, 2\*j\} \in \pi$ for some pair $(i,j)$ then also $\{2\*i, 2\*j-1\} \in \pi$) and/or four-point 
blocks $\{2\*i-1, 2\*i, 2\*j-1, 2\*j\}, \ \ 1\leq i<j\leq 2\*n.$
If $\pi$ is not optimal, it will be called suboptimal.
\end{definition}

In other words, the only blocks of an optimal partition $\pi$ are of the form $\{k_i, -k_j\}, \{-k_i, k_j\}$ 
(if one of such two-element sets appears in $\pi$ then the other must appear as well) or $\{k_i, -k_i, k_j, -k_j\}, \ i\neq j.$

If $\pi$ is optimal, then it induces a partition of
the set $\{1,\ldots 2\*n\}$ into pairs $\{i,j\}.$ Moreover, the subsum $\Sigma_{\pi}$ then factorizes as a product of
$n$ two-dimensional sums. Each sum corresponds to a pair $\{i,j\}$ and is proportional to $N,$ with the computations being identical to the ones 
discussed in the $l=2$ case above.  In particular, the coefficient in front of $N$ is given by a Riemann sum of the integral
$\frac{1}{\pi}\* \int_{\mathbb{R}} |\hat{f}(t)|^2 \*\min(|t|,1)^2\* dt$ in the case of paired two-point blocks 
$\{2\*i-1, 2\*j\}, \ \{2\*i, 2\*j-1\}, \ \ 1\leq i<j\leq 2\*n,$ and is equal to a Riemann sum of
\[ -\frac{1}{\pi}\* \int_{|s-t|\leq 1, |s|\vee|t|\geq 1} \hat{f}(t)\*\hat{f}(s)\*
(1-|s-t|)\* ds\*dt
-\frac{1}{\pi}\* \int_{0\leq s,t\leq 1, s+t>1} \hat{f}(s)\*\hat{f}(t)\*(s+t-1) \*ds\*dt\]
in the case of a four-point 
block $\{2\*i-1, 2\*i, 2\*j-1, 2\*j\}, \ \ 1\leq i<j\leq 2\*n.$

The main combinatorial ingredient of the proof of Theorem 2.5 is the following lemma that shows that suboptimal partitions give vanishing contributions
to the moments of the normalized random variable $\frac{S_N(f)-\E S_N(f)}{\sqrt{\var S_N(f)}}.$

\begin{lemma}
Let $\pi$ be a suboptimal partition of $\{1,2,\ldots, 2\*l\}=\{k_1, -k_1, \ldots, k_l, -k_l\}.$
Then the corresponding subsum 
$\Sigma_{\pi}$ is much smaller than $N^{l/2}$ in the limit $N \to \infty.$
In other words, 
\[ \Sigma_{\pi} \*N^{-l/2} \to 0
\]
for any suboptimal $\pi.$
\end{lemma}

The result of Theorem 2.5 then immediately follows from Lemma 7.2 and $l=2$ (variance) computations since they imply 
that the moments of $\frac{S_N(f)-\E S_N(f)}{\sqrt{\var S_N(f)}}$ converge
in the limit $N\to \infty$ to the moments of the standard Gaussian distribution. Indeed, combining all optimal subsums $\Sigma_{\pi},$ 
we conclude that the $2\*n$-th moment of
$\frac{S_N(f)-\E S_N(f)}{\sqrt{\var S_N(f)}}$ converges to $(2n-1)!!$ in the limit $N\to \infty.$

To prove Lemma 7.2, we recall the results of Lemma 5.2 about joint cumulants of traces of powers of a CUE  random matrix. The parts (i) and (ii) of 
Lemma 5.2 are of particular importance in our analysis. 

Let $\pi$ be a partition of the set $\{1,\ldots 4\*n \}$ that
has no singletons and no two-point subsets of the form
$\{2\*i-1, 2\*i\}, \ \ i=1,\ldots, 2n$ (as required by Lemma 9.2) and is not optimal. We have to show that $\Sigma_{\pi}\*N^{-n} \to 0.$
We proceed by induction in $n.$

First, without loss of generality, we can assume that $\pi$ does not contain paired two-point blocks 
$\{2\*i-1, 2\*j\}, \ \{2\*i, 2\*j-1\} $ and four-point 
blocks $\{2\*i-1, 2\*i, 2\*j-1, 2\*j\}.$ Indeed, if it does contain one of those, the subsum $\Sigma_{\pi}$ factorizes into the sum corresponding to
variables $k_i$ and $k_j$ and the sum corresponding to the remaining variables.
The first sum is proportional to $N.$ The second sum corresponds to a partition $\pi'$ of $4\*n-4$ element set, 
where $\pi'$ is obtained from $\pi$ by removing the above-mentioned block(s) 
corresponding to the $(i,j)$ pair (i.e. removing variables $k_i, -k_i$ and $k_j, -k_j$). 
Applying the induction assumption to $\Sigma_{\pi'}$ finishes the argument.

By the same token, we may assume that $\pi$ does not contain paired three-point blocks 
corresponding to variables $\{k_i, k_j, -k_p\}, \ \ \{-k_i, -k_j, k_p\}, \ \ k_i+k_j=k_p.$ 
If such paired three-point blocks belong to $\pi$ for some 
triple $(i,j,p),$ then the subsum 
$\Sigma_{\pi}$ again factorizes, and the 
sum corresponding to $\{k_i, k_j, -k_p\}, \ \ \{-k_i, -k_j, k_p\}, \ \ k_i+k_j=k_p$ is proportional to $N$ 
as was shown in the $l=3$ case computations above. 
Considering the partition $\pi'$ of a $4\*n-6$ element set obtained from $\pi$ by removing the paired three-point 
blocks $\{k_i, k_j, -k_p\}, \ \ \{-k_i, -k_j, k_p\}$ and applying the induction argument
to $\Sigma_{\pi'}$ finishes the argument.

In addition, we may also assume that $\pi$ does not contain a pair (two-element block) corresponding to variables 
$\{k_i, -k_j\}, \ i\neq j.$ If $\pi$ contains such a two-element block then
$k_i=k_j$ since otherwise $\kappa_2^{(N)}(k_i, -k_j)$ vanishes. Without loss of generality, we may assume $i=1$ and $j=2$. Consider the blocks containing
variables $-k_1$ and $k_2$ correspondingly, namely $\{-k_1, \epsilon_{i_1}\*k_{i_1}, \ldots, \epsilon_{i_m}\*k_{i_m}\}$ and 
$\{k_2, \epsilon_{j_1}\* k_{j_1}, \ldots \epsilon_{j_r}\* k_{j_r}\},$ where each $\epsilon=\pm 1.$
Then, instead of the original partition, consider a modified one denoted by $\pi''.$  The new partition $\pi''$ 
contains the blocks $\{k_1, -k_2\}, \ \{-k_1, k_2\}, 
\ \{\epsilon_{i_1}\*k_{i_1}, \ldots, \epsilon_{i_m}\*k_{i_m}, \epsilon_{j_1}\* k_{j_1}, \ldots \epsilon_{j_r}\*k_{j_r}\},$ and 
all the remaining blocks of $\pi.$ In other words, we replace three blocks $\{k_1, -k_2\}, \ 
\{-k_1, \epsilon_{i_1}\*k_{i_1}, \ldots, \epsilon_{i_m}\*k_{i_m}\}, \ \{k_2, \epsilon_{j_1}\* k_{j_1}, \ldots \epsilon_{j_r}\* k_{j_r}\}$ by three blocks
$\{k_1, -k_2\}, \ \{-k_1, k_2\}, 
\ \{\epsilon_{i_1}\*k_{i_1}, \ldots, \epsilon_{i_m}\*k_{i_m}, \epsilon_{j_1}\* k_{j_1}, \ldots \epsilon_{j_r}\*k_{j_r}\}.$
By power counting,  the subsum $\Sigma_{\pi''}$ corresponding to the 
the modified partition is of higher order in $N.$ 

One special case requires a separate treatment here. Indeed, when $\pi$ contains the blocks $\{k_1, -k_2\}, \ 
\{-k_1, k_j\},$ and $\{k_2, -k_j\},$ the modified partition $\pi''$ would contain the two-point block $\{k_j, -k_j\}$ forbidden by Lemma 9.2. 
Thus, the induction assumption does not apply. However, in such a case
$\Sigma_{\pi}$ clearly factorizes. The three-dimensional sum corresponding to variables $k_1, k_2, k_3$ 
has been studied earlier in the $l=3$ case. It has been
shown  to be proportional to $N.$ One then applies the induction assumption to the remaining $(2n-3)$-dimensional sum (and a corresponding partition of a 
$4n-6$ element set.)

Finally, we can assume that a partition $\pi$ does not contain a three-point subset $\{k_i, k_j, -k_m\}.$ Indeed, suppose $\pi$ contains, say,
$\{k_1, k_2, -k_3\}.$  Consider blocks, containing the variables $-k_1, -k_2$ and $k_3, $ namely
$\{-k_1, \epsilon_{i_1}\*k_{i_1}, \ldots, \epsilon_{i_m}\*k_{i_m}\},  \  \{-k_2, \epsilon_{j_1}\* k_{j_1}, \ldots \epsilon_{j_r}\* k_{j_r}\},$ and \\
$\{k_3, \epsilon_{g_1}\* k_{g_1}, \ldots \epsilon_{g_s}\* k_{g_s}\},$ correspondingly, where each $\epsilon=\pm 1.$ Compare $\pi$ with a modified partition
$\pi''$ that contains blocks $\{k_1, k_2, -k_3\}, \ \{-k_1, -k_2, k_3\}, \ \{\epsilon_{i_1}\*k_{i_1}, \ldots, \epsilon_{i_m}\*k_{i_m},
\epsilon_{j_1}\* k_{j_1}, \ldots \epsilon_{j_r}\* k_{j_r}, \epsilon_{g_1}\* k_{g_1}, \ldots \epsilon_{g_s}\* k_{g_s}\},$ and all the remaining blocks of 
$\pi.$ Again, by power counting, the subsum $\Sigma_{\pi''}$ corresponding to the the modified partition is of higher order in $N.$ 

Now, we are ready to finish the proof of the lemma. It remains to consider the case of a suboptimal partition $\pi$ of $\{1,2,\ldots 2\*l\}$
such that all blocks of $\pi$ consist of at least four elements. If $\pi$ contains a subset of cardinality $5$ or higher, then 
the number of blocks of the partition is not bigger than $l/2-1.$ Each cumulant in (\ref{cenmoments1}) is $O(N)$ by Lemma 5.2 (i).
Thus, the subsum $\Sigma_{\pi}$ of (\ref{7.2}) is bounded from above by 
\[
C\* N^{-l}\* N^{l/2-1}\* \sum_{k_1,\ldots k_l\geq 1} \prod_{i=1}^l |\hat{f}(k_i/N)|=O(N^{l/2-1}),
\]
where $C$ is a constant independent of $N.$
If all blocks of a partition $\pi$ have cardinality $4,$ then the number of blocks is $l/2.$ However, since $\pi$ is suboptimal, 
at least one of the blocks is not of the form $\{k_i, -k_i,k_j, -k_j\}.$
It follows then from Lemma 5.2 (ii) that the variables $k_1, k_2, \ldots, k_l$ are not linearly independent, and the subsum $\Sigma_{\pi}$ is bounded 
from above by

\[
C\* N^{-l}\* N^{l/2}\* \sum^*_{k_1,\ldots k_l\geq 1} \prod_{i=1}^l |\hat{f}(k_i/N)|,
\]
where the sum in the above formula is over linearly dependent variables and is $O(N^{l-1}).$ This immediately implies that
\[ \Sigma_{\pi}=O(N^{l/2-1}).\]
The lemma is proven. This finishes the proof of Theorem 2.5.
Below, Appendices 1,2,and 3 contain some standard auxiliary results.

\section{Appendix 1}

        Here we finish the proof of Theorem 2.1.
        Let $\varphi_m$ be i.i.d $\exp(1)$ random variables and define:
        \begin{align}
        \label{8.1}
            \mathcal{T}_k&=\frac{4}{\beta}\sum_{m=1}^k\hat{f}(m)m(\varphi_m-1)\hspace{5mm}\text{and}\hspace{5mm}
            \mathcal{T}_{\infty}=\frac{4}{\beta}\sum_{m=1}^{\infty}\hat{f}(m)m(\varphi_m-1).
        \end{align}
        We wish to show that $S_N(f)$ (see (\ref{pairs}))
converges in distribution to $\mathcal{T}_{\infty}$ by verifying convergence in the L\'evy metric. 
In other words, we check that the following pair of inequalities hold for arbitrary $\delta>0$ and sufficiently large $N$:
        \begin{align}
        \label{8.2}
            \Pr\of{S_N(f)-\E S_N(f)\leq x}\leq \Pr\of{\mathcal{T}_{\infty}-\E \mathcal{T}_{\infty}\leq x+\delta}+\delta
        \end{align}   
        \begin{align}
        \label{8.3}
            \Pr\of{\mathcal{T}_{\infty}-\E \mathcal{T}_{\infty}\leq x-\delta}-\delta \leq \Pr\of{S_N(f)-\E S_N(f)\leq x}.
        \end{align}
        Let us first consider the case $\beta=2$. To verify (\ref{8.2}) we apply a trivial probability bound and Chebyshev's inequality to get
        \begin{align}
            \Pr\of{S_N(f)-\E S_N(f)\leq x}&\leq \Pr\of{S_N(f_k)-\E S_N(f_k)\leq x+\frac{\delta}{3}}\nonumber\\
            &\hspace{15mm}+\Pr\of{|S_N(f-f_k)-\E S_N(f-f_k)|>\frac{\delta}{3}}\nonumber\\
        \label{8.4}
            &\leq \Pr\of{S_N(f_k)-\E S_N(f_k)\leq x+\frac{\delta}{3}}+\frac{9\var(S_N(f-f_k))}{\delta^2}.
        \end{align}
        It follows immediately from Proposition 4.3 that, for  sufficiently large $k$ and $N$ we can bound the variance of the tail, $S_N(f-f_k)$, 
by an arbitrarily small quantity: 
            \begin{align*}
                \var(S_N(f-f_k))=\sum_{m=k+1}^{\infty}|\hat{f}(m)|^2|m|^2+o(1)= o(1).
            \end{align*}
        Now choose $K_0$ (uniformly in $N$) and $N_0$ (uniformly in $k$) large enough so that if $k\geq K_0$ and $ N\geq N_0$ the following inequality 
is satisfied:
        \begin{align}
        \label{8.5}
            \frac{9\var(S_N(f-f_k))}{\delta^2}=\frac{9}{\delta^2}o(1)+\frac{9}{\delta^2}o(1)\leq \frac{\delta}{3}.
        \end{align}
        Since $\mathcal{T}_k\xrightarrow{\hspace{1mm}\mathcal{D}\hspace{1mm}}\mathcal{T}_{\infty}$, they must converge in the L\'evy metric. 
We can thus choose $K_1$  such that, for all $k\geq K_1$,
            \begin{align}
                \label{8.6}
                \Pr\of{\mathcal{T}_{\infty}-\E \mathcal{T}_{\infty}\leq x-\frac{\delta}{3}}-\frac{\delta}{3}\leq 
\Pr\of{\mathcal{T}_{k}-\E \mathcal{T}_{k}\leq x}\leq \Pr\of{\mathcal{T}_{\infty}-\E \mathcal{T}_{\infty}\leq x+\frac{\delta}{3}}+\frac{\delta}{3}.
            \end{align}
        Similarly, by \cite{johansson1}, there is an $N_1$  such that if $N\geq N_1$ the following holds:
        \begin{align}
        \label{8.7}
            \Pr\of{\mathcal{T}_{k}-\E \mathcal{T}_{k}\leq x-\frac{\delta}{3}}-\frac{\delta}{3}&\leq \Pr\of{S_N(f_k)-\E S_N(f_k)\leq x}\leq 
\Pr\of{\mathcal{T}_{k}-\E \mathcal{T}_{k}\leq x+\frac{\delta}{3}}+\frac{\delta}{3}.
        \end{align}
        We observe that  $N_1$ may depend  on $k$ so we simply choose the $N_1$ associated to $(\max(K_0,K_1))$. Thus  we let 
$k\geq K=\max (K_0, K_1)$ and $N\geq N_2= \max (N_0, N_1(\max (K_0, K_1)))$. Combining the rightmost inequalities in (\ref{8.6}) and (\ref{8.7}) and 
replacing $x$ with $x+\frac{\delta}{3}$, we obtain the following bound for the first term of (\ref{8.4}): 
            \begin{align*}
                \Pr\of{S_N(f_k)-\E S_N(f_k)\leq x+\frac{\delta}{3}}&\leq \Pr\of{\mathcal{T}_k-\E \mathcal{T}_k\leq x+\frac{2\delta}{3}}+\frac{\delta}{3}\\
                &\leq \Pr\of{\mathcal{T}_{\infty}-\E \mathcal{T}_{\infty}\leq x+\delta}+\frac{2\delta}{3}.
            \end{align*}
        Finally, using (\ref{8.5}) to bound the variance in (\ref{8.4}), we obtain the final $\delta/3$ term needed to ensure the desired inequality:
            \[
                \Pr\of{S_N(f_k)-\E S_N(f_k)\leq x+\frac{\delta}{3}}\leq \Pr\of{\mathcal{T}_{\infty}-\E \mathcal{T}_{\infty}\leq x+\delta}+\delta
            \]
        Using the same $K$ and $N_2$ we can now verify (\ref{8.3}). Indeed, assuming $k\geq K$ and $N\geq N_2$ by (\ref{8.6}) and (\ref{8.7}), we have:
        \begin{align*}
            \Pr(\mathcal{T}_{\infty}-\E \mathcal{T}_{\infty}\leq x&-\delta)-\delta\\
            &\leq \Pr\of{\mathcal{T}_k-\E \mathcal{T}_k\leq x-\frac{2\delta}{3}}-\frac{2\delta}{3}\nonumber\\
            &\leq \Pr\of{S_N(f_k)-\E S_N(f_k)\leq x-\frac{\delta}{3}}-\frac{\delta}{3}\nonumber\\
            &\leq \Pr\of{S_N(f)-\E S_N(f)\leq x}+\Pr\of{|S_N(f-f_k)-\E S_N(f-f_k) |\geq \frac{\delta}{3}}-\frac{\delta}{3}\nonumber\\
            &\leq\Pr\of{S_N(f)-\E S_N(f)\leq x},
        \end{align*}
        where the last inequality follows from the bound given in (\ref{8.5}). This concludes the proof for the case $\beta=2$.
        
        If $\beta\neq 2$, then we replace the Chebyshev bound in (\ref{8.4}) with the corresponding Markov bound and apply the results of Jiang 
and Matsumoto \cite{jm}. To see this, we will first rewrite the tail as
        \begin{align}
\label{tailtail}
            S_N(f-f_k)-\E{S_N(f-f_k)}=\sum_{m=k+1}^{\infty}\hat{f}(m)\left(|t_{N,m}(\overbar{\theta})|^2-\E{|t_{N,m}(\overbar{\theta})|^2}\right),
        \end{align}
            where
            \[
                t_{N,m}(\overbar{\theta})=\sum_{j=1}^N e^{im\theta_j}.
            \]
        
         For $0<\beta<2$, the proof of Lemma 4.3 in \cite{jm} gives the bound $\E|t_{N,m}(\overbar{\theta})|^2\leq (2/\beta)m$ for all $m\geq 1$ and $N\geq 2$. It follows that
            \begin{align*}
                \Pr\of{|S_N(f-f_k)-\E S_N(f-f_k) |\geq \frac{\delta}{3}}&\leq \frac{3\E{|S_N(f-f_k)-\E S_N(f-f_k) |}}{\delta}\\
                &\leq \frac{3}{\delta}C\left(\sum_{m=k+1}^{\infty}|\hat{f}(m)| |m| \right),
            \end{align*}
        where $C$ is a constant independent of $N$. Applying the condition in Theorem 2.1 for $0<\beta<2$, the r.h.s. side of the above inequality vanishes asymptotically, independent of $N$.
    
        For $\beta=4$ we break up ( \ref{tailtail}) into three pieces:
        \begin{align*}
            S_N(f-f_k)-\E{S_N(f-f_k)} = \sum_{m=k+1}^N (*)+ \sum_{m=N+1}^{2N}(*) + \sum_{m=2N+1}^\infty(*).
        \end{align*}
        Proposition 2 in \cite{jm} states that there exist constants $C, K$, independent of $N$, such that $\E|p_m(\overbar{\theta})|^2\leq Cm$ in the first sum, $\E|p_m(\overbar{\theta})|^2\leq Km\log(m+1)$ in the second sum, and $\E|p_m(\overbar{\theta})|^2\leq 2K\*N$ in the third sum. This gives the following bound for any $\alpha>0$: 
            \begin{align*}
                |S_N(f-f_k&)-\E S_N(f-f_k) |\\ 
                    &\leq C'\left(\sum_{m=k+1}^N m\cdot|\hat{f}(m)| + \sum_{m=N+1}^{2N} m\log(m+1)|\hat{f}(m)| +  \sum_{m=2N+1}^\infty m\cdot|\hat{f}(m)|\right),
            \end{align*}
        where $C'$ is a constant independent of $k$ and $N$. Applying the condition in Theorem 2.1 for $\beta=4$, the first sum goes to zero in $k$ independent of $N$ and the last two sums go to zero in $N$ independent of $k$.
    
        When $2<\beta\neq 4$, we break the tail as follows:
        \begin{align*}
            S_N(&f-f_k)-\E{S_N(f-f_k)}\\
                       &=\sum_{m=k+1}^{N/2}\hat{f}(m)\left(|t_{N,m}(\overbar{\theta})|^2-\E{|t_{N,m}(\overbar{\theta})|^2}\right)+\sum_{m=N/2+1}^{\infty}\hat{f}(m)\left(|t_{N,m}(\overbar{\theta})|^2-\E{|t_{N,m}(\overbar{\theta})|^2}\right)
        \end{align*}
        In the first sum, $\E{|t_{N,m}(\theta_n)|^2}\leq Cm$ where $C=2/\beta$ for $0<\beta<2$ and $C= e^{1-2/\beta}$ for $\beta>2$. 
For the second sum, we use the trivial bound  $\E{|t_{N,m}(\overbar{\theta})|^2}\leq N^2$ to get
            \[
                \E{|S_N(f-f_k)-\E S_N(f-f_k) |}\leq \left(2C\sum_{m=k+1}^{N/2}|\hat{f}(m)| |m| + 
8\sum_{m=N/2+1}^{\infty}|\hat{f}(m)| |m|^2\right).
            \]
        Once again, the first sum goes to zero in $k$ independent of $N$ and the second sum goes to zero in $N$ independent of $k$.
  \end{proof}

\section{Appendix 2}

Denote
\begin{align*}
G(k_1, \ldots, k_n):=\sum_{\sigma \in S_n} \ \sum_{m=1}^n \frac{(-1)^{m-1}}{m} \*\sum_{\substack{(n_1, \ldots, n_m):
\\ n_1+\ldots n_m=n, \ n_1, \ldots n_m\geq 1}} 
\frac{1}{n_1!\cdots n_m!} \* \\
\max\left(0, \sum_i^{n_1} k_{\sigma(i)}, \sum_i^{n_1+n_2} k_{\sigma(i)}, \ldots, \sum_i^{n_1+\ldots+n_{m-1}} k_{\sigma(i)}\right).
\end{align*}

The following statement was proven in \cite{sasha}:
\begin{lemma}
Let $\sum_i k_i=0.$ Then $G(k_1, \ldots, k_n)$ equals zero for $n>2$ and $G(k,-k)=|k|$ for $n=2.$
\end{lemma}

The following standard lemma plays an important role in the CLT proof in the microscopic case (Theorem 2.5, Section 7).

\begin{lemma}
Let $ X_1, \ldots, X_{2n}$ be centered random variables with finite mathematical expectations. Then
\begin{align}
\label{cenmoments}
\E \prod_{1\leq i\leq n} (X_{2i-1}\*X_{2i}-\E X_{2i-1}\*X_{2i})   = \sum^*_{\pi} \prod_{B\in \pi} \kappa(X_i: i \in B),
\end{align}
where where the sum on the r.h.s. of (\ref{cenmoments}) is over all partitions $\pi$ of $\{1, \ldots, 2n\}$ that do not contain atoms and 
two-element subsets of the 
form $\{2i-1, 2i\}, \ i=1, \ldots,n.$
\end{lemma}

\begin{proof}

It follows from (\ref{moments}) that the r.h.s. of (\ref{cenmoments}) is equal to a linear combination of $\prod_{B\in \pi} \kappa(X_i: i \in B),$  where
$\pi$ runs over the list of partitions of $\{1,2,\ldots, 2n\}.$ Since $X_i$'s are centered, partitions $\pi$ with one-element subsets (atoms) give 
zero contribution.
If  $\pi$ does not contain a subset of the form  $\{2i-1, 2i\}, \ i=1, \ldots,n,$ then the coefficient in front of the product 
$ \prod_{B\in \pi} \kappa(X_i: i \in B)$ in the linear 
combination is $1$ since it comes from $\E \prod_{1\leq i\leq n} X_{2i-1}\*X_{2i}.$
Finally, suppose that $\pi$ contains $s$ two-elements subsets of the prescribed form, namely $\{2i_1-1, 2i_1\}, \ldots, 
\{2i_s-1, 2i_s\}, \ \ 1\leq s\leq n.$ Then the coefficient in front of  $\prod_{B\in \pi} \kappa(X_i: i \in B)$ is equal to  
\begin{align}
\sum_{k=0}^s (-1)^k \frac{s!}{k!\*(s-k)!}=0.
\end{align}
\end{proof}

For convenience of the reader, we finish this section with the proposition which is related to Lemma 1 and (2.8) from \cite{sasha}.
\begin{prop}
\begin{align}
   & & \kappa(T_N(f_1), \ldots, T_N(f_n))=\sum_{m=1}^n\sum_{\substack{\text{ordered collections} \\\text{of subsets }
\mathcal{R}=\{R_1,..R_m\}}}\frac{(-1)^{m-1}}{m}\sum_{k_1+...+k_N=0}\nonumber \\
& & 
\widehat{f_{j_1^1}}(k_1)...\widehat{f_{j_{l_1}^1}}(k_{l_1})\widehat{f_{j_1^2}}(k_{i_1+1})....
\widehat{f_{j_{l_2}^2}}(k_{l_1+l_2})
   \widehat{f_{j_1^m}}(k_{l_1+l_2+...l_{m-1}+1})...\widehat{f_{j_{l_m}^m}}(k_N)\nonumber\\
& & 
    \times \#\{u:0\leq u\leq N-1, 0\leq u+\sum_{i=1}^{l_1}k_i\leq N-1,...0\leq u+\sum_{i=1}^{l_1+...l_{m-1}}k_i\leq N-1\},
\end{align}
where the sum is over all ordered collections of subsets $\mathcal{R}=\{R_1,..R_m\}$ such that $\bigsqcup_{1\leq i\leq m} R_i=\{1,2,\ldots, N\},$
and $R_1=\{j_1^1, \ldots, j_{l_1}^1\}, \ \ R_2=\{j_1^2, \ldots, j_{l_2}^2\}, \ldots, \ R_m=\{j_1^m, \ldots, j_{l_m}^m\}.$
\end{prop}
\begin{proof}
We start by computing the joint moment of linear CUE statistics.
\begin{align*}  E_{\{1,2,...,n\}}=\E \of{\sum_{i_1=1}^N f_{1}(\theta_{i_1})\cdot ...\cdot\sum_{i_n =1}^N f_n(\theta_{i_n})}\end{align*}

Let $\mathcal{M}$ be a partition of $\{1,2,...n\}$ into subsets determined by coinciding indices in $i_1,..., i_n$. The above mixed moment 
can be rewritten as:

\begin{align}\label{cum:1}
E_{\{1,2,...,n\}}=\sum_{\substack{\text{partitions} \\\mathcal{M}=\{M_1,..,M_r\}\\\sqcup M_i=\{1,...,n\}}}
\E \left (\sum_{l_1\neq...\neq l_m}f_{M_1}(\theta_{l_1})\cdot...\cdot f_{M_r}(\theta_{l_r})\right)\end{align}

Here   $f_M(\theta)=\prod_{j\in M}f_j(\theta)$ . To compute the expectations, we use the determinantal 
structure of the CUE point-correlation functions. Indeed,
\[
\E_{\text{CUE}(N)}\sum_{l_1\neq...\neq l_r}f_{M_1}(\theta_{l_1})\cdot...\cdot f_{M_r}(\theta_{l_r})=
\int_{[0,2\pi]^r} f_{M_1}(x_1)\cdot...\cdot f_{M_r}(x_r)\*\rho_r(x_1,\ldots, x_r) \ dx_1\ldots dx_r,
\]
where
\begin{align}
\label{corfunction}
\rho_r(x_1, \ldots, x_r)=\det\left(Q_N(x_i,x_j)_{i,j=1,\ldots,r}\right),\ \ 
Q_N(x,y)=\frac{1}{2\*\pi}\*\sum_{j=0}^{N-1} e^{i\*j\*(x-y)}. 
\end{align}

Writing
\begin{align}
 \label{cum:2}   \rho_r(\theta_1,&...,\theta_r)=\sum_{\sigma\in S_r}(-1)^{|\sigma|}\prod_{i=1}^{r}Q_N(\theta_i, \theta_{\sigma(i)})\nonumber \\
    &=\sum_{\substack{\text{partitions } \mathcal{K}\\ 
\sqcup K_{\alpha}=\{1,...,r\}}}\of{\prod_{\alpha=1}^q (-1)^{p_{\alpha}-1}\sum_{\substack{\text{cyclic perms}\\
\text{of $K_{\alpha}$}}}\prod_{j=1}^{p_{\alpha}}Q_N\of{\theta_{t_j^{(\alpha)}},\theta_{\sigma(t_j^{(\alpha)})}}}
\end{align}
   
   In the second equality we wrote the permutation $\sigma\in S_r$ as a product of cycles. This partitions $\{1,2,...,r\}=
\sqcup_{\alpha=1}^qK_{\alpha}=\sqcup_{\alpha=1}^q\{t_1^{(\alpha)},...,t_{p_{\alpha}}^{(\alpha)}\}$ into supports of those cycles. 
The expression resulting from computing the 
expectations in (\ref{cum:1}) using (\ref{corfunction}-\ref{cum:2}) can be simplified by  defining a new partition . 
Let $\mathcal{P}=\sqcup_{i=1}^q P_i$ where  $P_i=\sqcup_{j\in K_i}M_j$. Now observe that $\mathcal{P}_i=\sqcup_{j\in K_i}M_j$ induces a 
partition of each $P_i$. Thus, exchanging summation:

    \begin{align}E_{\{1,2,...,n\}}&=\sum_{\substack{\text{partitions } \mathcal{P}\\ \text{of }\{1,2...,n\}}}
\prod_{i=1}^q \sum_{\substack{\text{partitions } \mathcal{P}_i\\ \text{of }\{P_{i,1},...,P_{i,t_i}\}}}
\int_{\mathbb{T}^{t_i}}f_{P_{i,1}}(\theta_1)\cdot...\cdot f_{P_{i,t_{i}}}(\theta_{t_i})\nonumber \\
    &\label{cum:3}\times(-1)^{t_i-1}\sum_{\substack{\text{cyclic perms}\\ \sigma\in S_{t_i}}}\prod_{j=1}^{t_i}Q_N(\theta_j,\theta_{\sigma(j)})\end{align}

Recall that joint cumulants and joint moments are related by the following formula:

\begin{align*}
E_{\{1,2,...N\}}=\sum_{\substack{\text{partitions} \\\mathcal{M}=\{M_1,..,M_r\}\\\sqcup M_i=\{1,...,N\}}} \kappa_{M_1}\cdot...\cdot 
\kappa_{M_r}=\end{align*}

Comparing it with (\ref{cum:3}) we can express the joint  cumulants with indices $\{i_1,...i_l\}$ as:
\begin{align*}
\kappa_{\{i_1,...,i_l\}} =\sum_{m=1}^l (-1)^{m-1} \sum_{\substack{\text{partitions} \\\mathcal{R}=\{R_1,..R_m\}\\
\sqcup R_i=\{i_1,...i_l\}}}\int_{\mathbb{T}^m}f_{R_1}(\theta_1)...f_{R_m}(\theta_m) \sum_{\substack{\text{cyclic}\\
\text{permutations}\\
\sigma \in S_m}}\prod_{j=1}^m Q_N(\theta_j,\theta_{\sigma(j)})d\theta_1...d\theta_m\end{align*}

We may replace the range of the inside sum  by averaging  over all permutations to obtain:

\begin{align*}\kappa_{\{i_1,...,i_l\}}=\sum_{m=1}^l \frac{(-1)^{m-1}}{m}\sum_{\substack{\text{partitions} 
\\\mathcal{R}=\{R_1,..R_m\}\\\sqcup R_i=\{i_1,...i_l\}}}
\int_{\mathbb{T}^m}f_{R_1}(\theta_1)...f_{R_m}(\theta_m) \sum_{\sigma\in S_m}\prod_{j=1}^m Q_N(\theta_{\sigma(j)},
\theta_{\sigma(j+1)})d\theta_1...d\theta_m\end{align*}

Next we observe that the change of variables $\theta_{\sigma(j)}\rightarrow \theta_j$ effectively amounts to permuting the elements of the $R_i$'s 
and hence;

\begin{align*}
\kappa_{\{i_1,...,i_l\}}&=\sum_{m=1}^l\sum_{\substack{\text{partitions} \\\mathcal{R}=\{R_1,..R_m\}\\\sqcup R_i=\{i_1,...i_l\}}}\sum_{\sigma\in S_m} 
\int_{\mathbb{T}^m}f_{\sigma(R_1)}(\theta_1)...f_{\sigma(R_m)}(\theta_m) 
\frac{(-1)^{m-1}}{m}\prod_{j=1}^m Q_N(\theta_{j},\theta_{j+1})d\theta_1...d\theta_m\\
&=\sum_{m=1}^l\sum_{\substack{\text{ordered collections} \\\text{of subsets } \mathcal{R}=\{R_1,..R_m\}}} 
\int_{\mathbb{T}^m}f_{R_1}(\theta_1)...f_{R_m}(\theta_m) \frac{(-1)^{m-1}}{m}\prod_{j=1}^m Q_N(\theta_{j},\theta_{j+1})d\theta_1...d\theta_m
\end{align*}

Finally, we integrate to obtain the following expression in terms of Fourier coefficients:

\begin{align}\label{cum:4}\kappa_{\{i_1,...,i_l\}}=\sum_{m=1}^l\sum_{\substack{\text{ordered collections} 
\\\text{of subsets }\\ \mathcal{R}=\{R_1,..R_m\}}} \frac{(-1)^{m-1}}{m}\sum_{s_1=0}^{N-1}...
\sum_{s_m=0}^{N-1}\hat{f}_{R_1}(-s_m+s_1)\hat{f}_{R_2}(-s_1+s_2)\cdot...\cdot\hat{f}_{R_m}(-s_{m-1}+s_m)
\end{align}

The Fourier coefficients $\hat{f}_{R_i}(-s_{i-1}+s_i)$ can be expanded as convolutions of the form:

\begin{align*}
   & \hat{f}_{R_1}(-s_m+s_1)=\sum_{\substack{(k_1,...,k_{|R_1|)}\\ k_1+k_2...+k_{|R_1|}
=-s_m+s_1}}\hat{f}_{1,1}(k_1)\cdot...\cdot \hat{f}_{1,|R_1|}(k_{|R_1|})\\
 & \quad \quad \quad \quad \quad \quad \quad \quad \quad \quad  \vdots\\
  & \quad \quad \quad \quad \quad \quad \quad \quad \quad \quad  \vdots\\
 & \hat{f}_{R_m}(-s_{m-1}+s_m)=\sum_{\substack{(k_{|R_1|+...+|R_{m-1}|+1}, ... ,k_l)\\ k_{|R_1|+...+|R_{m-1}|+1}...+k_{l}
=-s_m+s_1}}\hat{f}_{1,m}(k_{|R_1|+...+|R_{m-1}|+1})\cdot...\cdot \hat{f}_{1,|R_1|}(k_{l})
\end{align*}

Note $\sum_{1}^{l}k_i=0$ for the every term in the product of these convolutions. Counting over all possible possible  
$0\leq s_j=\sum_{i=1}^{l_1+...l_{j-1}}k_i\leq n-1$ from  (\ref{cum:4}) we arrive at the final explicit expression for the cumulants:

\begin{align}
   \kappa_{\{i_1,...,i_l\}}&=\sum_{m=1}^l\sum_{\substack{\text{ordered collections} \\\text{of subsets }\mathcal{R}=\{R_1,..R_m\}}}
\frac{(-1)^{m-1}}{m}\sum_{k_1+...+k_l=0}\hat{f_{j_1}}(k_1)...\hat{f}_{k_l}(k_l)\nonumber\\
    & \#\{u:0\leq u\leq N-1, 0\leq u+\sum_{i=1}^{l_1}k_i\leq N-1,...0\leq u+\sum_{i=1}^{l_1+...l_{m-1}}k_i\leq N-1\}
\end{align}

\end{proof}

\section{Appendix 3}
This appendix provides the modifications to Proposition 4.1 and Lemma 4.4  that are necessary to adapt the proof of Proposition 4.3 to the 
mesoscopic case $1\ll L_N\ll N.$
        
            \begin{lemma}(Extension of Proposition 4.1)\\
            If $\beta=2$ and $f\in C_c^\infty(\mathbb{R})$ is an even, smooth, compactly supported function on the real line, then, 
for sufficiently large $N$,
                \begin{align*}
                    \left(\frac{\pi}{2}\right)\var&\left(\frac{S_N(f(L_N\cdot))}{\sqrt{L_N}}\right) =\\
                    &\frac{1}{L_N}\sum_{1\leq s\leq N-1}\left(\frac{s}{L_N}\right)^2\left(\hat{f}
\left(\frac{s}{L_N}\right)\right)^2 + N^2\sum_{N\leq s} \left(\hat{f}\left(\frac{s}{L_N}\right)\right)^2\\ 
                    &- N\sum_{\substack{N\leq s}}\left(\hat{f}\left(\frac{s}{L_N}\right)\right)^2  
                    -\left(\frac{1}{L_N}\right)^2\sum_{\substack{1\leq s,t \\ 1\leq |s-t|\leq N-1\\ N\leq \max(s,t)}}
\left(\frac{N-|s-t|}{L_N}\right)\hat{f}\left(\frac{s}{L_N}\right)\hat{f}\left(\frac{t}{L_N}\right)\\
                    &-\left(\frac{1}{L_N}\right)^2\sum_{\substack{1\leq s,t\leq N-1\\N+1\leq s+t}}\left(\frac{(s+t)-N}{L_N}\right) 
\hat{f}\left(\frac{s}{L_N}\right)\hat{f}\left(\frac{t}{L_N}\right).
                \end{align*}
            \end{lemma}
            
            \begin{proof}$\\$
            If $f\in C_c^\infty(\mathbb{R})$, then we may assume $N$ large enough so that the support of $f(L_N\cdot)$ is contained on $[-\pi,\pi)$. 
We can immediately express $f(L_N\cdot)$ as a Fourier Series with coefficients determined by the Fourier transform of $f$. Lemma 10.1 is then an 
immediate corollary to Proposition 4.1. 
            \end{proof}

            \begin{lemma}(Extension of Lemma 4.4)$\\$
            Let $f\in C_c^\infty(\mathbb{R})$. Then
                \begin{enumerate}
                    \item[(i)]
                        \begin{align*}
                            \left(\frac{1}{L_N}\right)^2\sum_{\substack{1\leq s,t\leq N\\ s+t\geq N+1}}\left(\frac{s}{L_N}\right)\left|\hat{f}
\left(\frac{s}{L_N}\right)\right|\cdot\left|\hat{f}\left(\frac{t}{L_N}\right)\right|\to 0;
                        \end{align*}
                    \item[(ii)]
                        \begin{align*}
                            \frac{N+1}{L_N^3}\sum_{\substack{s-t\leq N\\s\geq N+1\\1\leq t\leq N}}
\left|\hat{f}\left(\frac{s}{L_N}\right)\right|\cdot\left|\hat{f}\left(\frac{t}{L_N}\right)\right|\to 0;
                        \end{align*}    
                    \item[(iii)]
                        \begin{align*}
                            \frac{N}{L_N^3}\sum_{\substack{|s-t|\leq N-1\\s,t\geq N}}\left|\hat{f}\left(\frac{s}{L_N}\right)\right|
\cdot\left|\hat{f}\left(\frac{t}{L_N}\right)\right|\to 0.
                        \end{align*}
                \end{enumerate}
            \end{lemma}
        \begin{proof}
            To see (i), we replace the Fourier coefficients in the the proof of Lemma 4.4(i) with  the corresponding coefficients for the scaled case to get
                \begin{align}
                    \left(\frac{1}{L_N}\right)^2&\sum_{\substack{1\leq s,t\leq N\\ s+t\geq N+1}}\left(\frac{s}{L_N}\right)
\left|\hat{f}\left(\frac{s}{L_N}\right)\right|\cdot\left|\hat{f}\left(\frac{t}{L_N}\right)\right|\nonumber\\
                \label{10.1}
                    &\leq 3\left[\frac{1}{L_N}\sum_{s=kL_N+1}^\infty 
\left(\frac{s}{L_N}\right)^2\left|\hat{f}\left(\frac{s}{L_N}\right)\right|^2 \right]^{1/2}\left[\frac{1}{L_N}
\sum_{s=1}^\infty \left(\frac{s}{L_N}\right)^2\left|\hat{f}\left(\frac{s}{L_N}\right)\right|^2 \right]^{1/2}\\
                    &\hspace{15mm}+  \frac{1}{L_N}\sum_{s=1}^{kL_N} \left|\hat{f}\left(\frac{s}{L_N}\right)\right|\cdot\left|
\frac{x_{N-s+1}+\dots +x_{N}}{L_N}\right|,\nonumber
                \end{align}
            where $k\in \mathbb{N}$. The first term in (\ref{10.1}) contains, in brackets, two Riemann Sums and consequently converges to
                \[
                    3\left(\int_k^\infty [x\hat{f}(x)]^2dx\right)^{1/2} \left(\int_0^\infty [x\hat{f}(x)]^2dx\right)^{1/2}=o_k(1).
                \]
            Since $f\in C_c^\infty(\mathbb{R})$, we can write $|\hat{f}(x)|\leq C'/x^2$ for some positive constant $C'$ depending only on $f$, i.e. 
independent of $k$ and $N$. It follows that
                \[
                    \frac{1}{L_N}\sum_{s=1}^{kL_N} \left|\hat{f}\left(\frac{s}{L_N}\right)\right|\cdot\left|\frac{x_{N-s+1}+\dots +x_{N}}{L_N}\right|
                    \leq\frac{C'}{L_N}\sum_{s=1}^{kL_N} \left|\hat{f}\left(\frac{s}{L_N}\right)\right|\cdot \left(\frac{1}{N-s+1}+\cdots+ \frac{1}{N}\right)
                \]
                \[
                    \leq \left(C'\frac{kL_N}{N-kL_N}\right)\left(\frac{1}{L_N}\sum_{s=1}^{kL_N} \left|\hat{f}\left(\frac{s}{L_N}\right)\right|\right).
                \]
            For any fixed $k$, the term on the left is $O\left(\frac{L_N}{N}\right)$ while the term on the right is a Riemann Sum converging to 
                \[
                    \int_0^k |\hat{f}(x)|\hspace{2mm}dx \leq ||\hat{f}||_1 <\infty
                \]
             as $N\to \infty$. It follows immediately that, for any $\epsilon>0$, we can choose $k$ and $N$ large enough so that both terms 
in (\ref{10.1}) are at most $\epsilon/2$. This gives the desired result.\\
            
            \begin{remark}
                In the above proof, we did not fully utilize the smoothness constraint on $f$. In fact, it would have been sufficient to 
have $f\in C^2_c(\mathbb{R})$.
            \end{remark}
            
            To see (ii), we observe that, in the same way as in the proof of (i), the proof of Lemma 4.4(ii) immediately implies 
                \begin{align}
                    \frac{N+1}{L_N^3}&\sum_{\substack{s-t\leq N\\s\geq N+1\\1\leq t\leq N}}\left|\hat{f}\left(\frac{s}{L_N}\right)\right|\cdot\left|
\hat{f}\left(\frac{t}{L_N}\right)\right|\nonumber\\
                \label{10.2}
                    &\leq 3\left(\frac{1}{L_N}\sum_{s=kL_N+1}^\infty \left(\frac{s}{L_N}\right)^2\left(\hat{f}
\left(\frac{s}{L_N}\right)\right)^2 \right)^{1/2}\left(\frac{1}{L_N}\sum_{s=1}^\infty \left(\frac{s}{L_N}\right)^2
\left(\hat{f}\left(\frac{s}{L_N}\right)\right)^2 \right)^{1/2}\\
                    &\hspace{15mm}+ \frac{1}{L_N}\left|\sum_{s=1}^{kL_N}\hat{f}\left(\frac{s}{L_N}\right)\frac{x_{N+1}+\dots + 
x_{N+s}}{L_N}\right|\nonumber
                \end{align}
            The first term in (\ref{10.2}) is the same as the first term in (\ref{10.1}). Similarly, we observe that the second term is bounded above by 
                \[
                    \left(C'\frac{L_N}{N}\right)\left(\frac{1}{L_N}\sum_{s=1}^{kL_N}
\left|\hat{f}\left(\frac{s}{L_N}\right)\right|\right)=O\left(\frac{L_N}{N}\right)\to 0.
                \]
            This completes the proof of (ii).\\
        
            To see (iii), we once again follow the same argument as in the proof of Lemma 4.4(iii). In particular, we split up the sum into two parts:
                \[
                 \left[\frac{N}{L_N^3}\sum_{\substack{t-N+1
                 \leq s\leq N+t-1\\t\geq N}}\left|\hat{f}\left(\frac{s}{L_N}\right)\right|\cdot\left|\hat{f}\left(\frac{t}{L_N}\right)\right|\right] - 
\left[\frac{N}{L_N^3}\sum_{\substack{t-N+1\leq s \leq N-1\\t\geq N}}\left|\hat{f}
\left(\frac{s}{L_N}\right)\right|\cdot\left|\hat{f}\left(\frac{t}{L_N}\right)\right|\right].
                \]
            The first sum goes to zero by (ii), while the proof of Lemma 4.4(iii) implies that the second sum is bounded above by 
                \[
                    3\left(\frac{1}{L_N}\sum_{s=N}^\infty \left(\frac{s}{L_N}\hat{f}\left(\frac{s}{L_N}\right)\right)^2\right)\left(\frac{1}{L_N}
\sum_{s=1}^\infty \left(\frac{s}{L_N}\hat{f}\left(\frac{s}{L_N}\right)\right)^2\right).
                \]
            The term on the r.h.s. is a Riemann sum that converges to 
                \[
                    \int_0^\infty (x\hat{f}(x))^2dx <\infty
                \]
            as $N\to \infty$, while the term on the l.h.s. is, at most, on the order of
                \[
                    \int_k^\infty (x\hat{f}(x))^2dx
                \]
            for any $k\in \mathbb{N}$, i.e. goes to zero as $N\to\infty$. This completes the proof of Lemma 10.2.\\
        \end{proof}



\newpage           
\begin{thebibliography}{99} 

\bibitem{AS} Aguirre, A., Soshnikov, A., \textit{A note on pair dependent linear statistics with slowly growing variance.}
in preparation.\\

\bibitem{BF} Baker, T.H., Forrester, P. J., \textit{Finite-N Fluctuation Formulas for Random Matrices.}
J. Stat. Phys., \textbf{88}, (1997), 1371--1386.\\


\bibitem{BL} Bekerman, F., Lodhia, A., \textit{Mesoscopic Central Limit Theorem for general $\beta$-ensembles.}
Ann. Inst. H. Poincare Prob. Stat. \textbf{54}, (2018), 1917--1938.

\bibitem{DE} Diaconis, P., Evans, S.N., \textit{Linear functionals of eigenvalues of random matrices.} Trans. Amer. Math. Soc.
\textbf{353}, (2001),2615--2633.\\

\bibitem{DS} Diaconis, P., Shahshhani, M. \textit{ On eigenvalues of random matrices.} J. Appl. Probab., 
\textbf{31A}, (1994), 49--62.

\bibitem{dur} Durrett, R. \textit{Probability. Theory and Examples.}
Cambridge University Press, 4th ed., 2010.

\bibitem{Dyson1} Dyson, F.J. \textit{Statistical theory of the energy levels of complex systems. I} J. Math. Phys.
\textbf{3}, (1962), 140--156.

\bibitem{Dyson2} Dyson, F.J. \textit{Statistical theory of the energy levels of complex systems. I} J. Math. Phys.
\textbf{3}, (1962),166--175.

\bibitem{Dyson3} Dyson, F.J. \textit{Statistical theory of the energy levels of complex systems. I} J. Math. Phys.
\textbf{3}, (1962), 1191--1198.

\bibitem{erdosyau} 
Erdos, L., Yau, H.T.  
\textit{ Dynamical Approach to Random Matrix Theory}. 
Courant Lecture Notes in Mathematics, 2017.\\

\bibitem{FTW} Feng, R., Tian, G., Wei, D., \textit{Normality of Circular $\beta$-Ensemble.}
available at arXiv:1905.09448 math.PR\\

\bibitem{forrester}
Forrester, P.J.
\textit{Log-Gases and Random Matrices}.
London Mathematical Society Monographs Series
\textbf{34}, Princeton Univ. Press, Princeton, 2010.




\bibitem{jm} Jiang, T., Matsumoto, S.  \textit{Moments of Traces of Circular $\beta$-ensembles.} Ann. Probab.
\textbf{43}, Number 6 (2015), 3279--3336\\

\bibitem{johansson1} 
Johansson, K. 
\textit{On Szego's Asymptotic Formula for Toeplitz Determinants and Generalizations}.  
Duke Math. J.
\textbf{91} (1988), 151--204.\\

\bibitem{johansson2} 
Johansson, K. 
\textit{On Fluctuations of Eigenvalues of Random Hermitian Matrices}.  
Duke Math. J.
\textbf{91} (1998), 151--204.\\

\bibitem{johansson3} 
Johansson, K. 
\textit{On Random Matrices from the Compact Classical Groups}.
Ann. Math (2)
\textbf{145}, (1997), 519--545.

\bibitem{HK} He, Y., Knowles, A., \textit{ Mesoscopic eigenvalue statistics of Wigner matrices},
Ann. Appl. Probab.  \textbf{27(3)}, (2017), 1510--1550.

\bibitem{HKOC}
Hughes, C.P., Keating, J.P.,  O'Connell, N. \textit{On  the  characteristic  polynomial of  a  random  unitary  matrix.}
Comm.  Math.  Phys., \textbf{220(2)}, (2001), 429--451.

\bibitem{KN}
Killip, R., Nenciu, I. \textit{Matrix models for circular ensembles.}, 
Int. Math. Res. Not. \textbf{50}, (2004), 2665--2701.\\

\bibitem{lambert} Lambert, G. \textit{Mesoscopic central limit theorem for the circular beta-ensembles and applications.}
available at arXiv:1902.06611 math.PR.\\

\bibitem{LSX}
Li, Y., Schnelli, K., Xu, Y.
\textit{Central limit theorem for mesoscopic eigenvalue statistics of deformed Wigner matrices and sample covariance matrices}
available at arXiv:1909.12821 math.PR.\\


\bibitem{LS}
Lodhia, A., Simm, N.J., \textit{ Mesoscopic linear statistics of Wigner matrices}, 
available at arXiv:1503.03533.\\


\bibitem{malyshev}
Malyshev, V.A., Minlos, R.A., \textit{Gibbs Random Fields. Cluster Expansions.} Springer, 1991.\\


\bibitem{meckes}
Meckes, E. S., Meckes, M. W. \textit{Self-similarity in the circular unitary ensemble.}
Discrete Anal. (2016), paper No.9, 14pp.\\

\bibitem{mehta}
Mehta, M.L.. \textit{Random Matrices}. Elsevier Ltd. , 2004\\

\bibitem{montgomery1}
Montgomery, H.L.
\textit{ On pair correlation of zeros of the zeta function.} Proc. Sympos. Pure Math., 
\textbf{24} , (1973), 181--193.\\

\bibitem{montgomery2}
Montgomery, H.L.
\textit{ Distribution of the zeros of the Riemann zeta function.},
Proc. Internat. Congr. Math.,
\textbf{1}, Vancouver, BC (1974), 379-381.\\

\bibitem{PZ}
Paquette, E., Zeitouni, O. \textit{The Maximum of the CUE Field.}
IMRN, \textbf{16}, (2018), 5028--5119\\

\bibitem{Rains}
Rains, E. \textit{High powers of random elements of compact Lie groups.} Probab. Theory Related Fields,
\textbf{107}, (1997), 219-241.\\


\bibitem{RS} Rudnick, Z., Sarnak, P., \textit{Zeros of principal L-functions and random matrix theory},
Duke Math. J. \textbf{81}, (1996), 269--322.

\bibitem{Sasha} 
Soshnikov, A. \textit{Level spacings distribution for large random matrices: Gaussian fluctuations.}
Ann. Math (2)
\textbf{148}, (1998), 573--617.\\

\bibitem{sasha} 
Soshnikov, A. \textit{Central Limit Theorem for local linear statistics in classical compact groups and related
combinatorial identities}. 
Ann. Probab. 
\textbf{28}, (2000), 1353--1370\\

\bibitem{tao} 
Tao, T.
\textit{Topics in Random Matrix Theory}. 
American Mathematical Society, 2012.\\

\bibitem{webb} 
Webb, C., \textit{Linear statistics of the circular ensemble, Stein's method, and circular Dyson Brownian motion.},
E. J. Probab., \textbf{20}, (2015), No. 104, 21pp.

\bibitem{WF1}
Witte N.S., Forrester, P.J.
\textit{Moments of the Gaussian Beta Ensembles Ensembles and the large-N expansion of the densities.}
J. Math. Phys., \textbf{55}, (2014), 083302
Electron. J. Probab. \textbf {21}, (2016), Paper No. 25, 16 pp. \\

\bibitem{WF2}
Witte N.S., Forrester, P.J.
\textit{ Loop equation analysis of the Circular Beta Ensembles.}
JHEP, \textbf{173}, (2015).\\

\end{thebibliography}
\end{document}